\documentclass[final,hidelinks,onefignum,onetabnum]{siamart251216}

\usepackage{lipsum}
\usepackage{amsfonts}
\usepackage{graphicx}
\usepackage{epstopdf}
\usepackage{booktabs}
\ifpdf
\DeclareGraphicsExtensions{.eps,.pdf,.png,.jpg}
\else
\DeclareGraphicsExtensions{.eps}
\fi

\usepackage{algorithm}
\usepackage{algpseudocode}


\usepackage{aliascnt}
\newsiamremark{remark}{Remark}
\crefname{remark}{remark}{remarks}
\Crefname{remark}{Remark}{Remarks}

\AddToHook{env/remark/begin}{\crefalias{theorem}{remark}}
\newsiamremark{assumption}{Assumption}
\crefname{assumption}{assumption}{assumptions}
\Crefname{assumption}{Assumption}{Assumptions}

\AddToHook{env/assumption/begin}{\crefalias{theorem}{assumption}}

\headers{Mode-Consistent Galerkin Closure}{H. Tomita}

\title{
  Mode-Consistent Galerkin Closure under Physical Energy Metric
  for Hyperbolic Systems
}
\author{
  Hirofumi Tomita\thanks{
    RIKEN Center for Computational Science, Kobe, Japan
    (\email{htomita@riken.jp}).
  }
}
\usepackage{amsopn}

\usepackage{mathtools,bm,enumitem}
\usepackage{amsfonts}  
\usepackage{amssymb}   

\ifpdf
\hypersetup{
  pdftitle={Mode-Consistent Galerkin Closure under Physical Energy Metric
    for Hyperbolic Systems},
  pdfauthor={Hirofumi Tomita}
}
\fi

\makeatletter

\newcommand{\hecf@pcutsub}{\smash{{[p]}}}
\newcommand{\hecf@ksup}{^{(k)}}
\newcommand{\hecf@osup}{^{\smash{\lower0.0ex\hbox{$\scriptstyle (o)$}}}}
\newcommand{\hecf@osupp}{^{\smash{\lower0.7ex\hbox{$\scriptstyle (o)$}}}}
\newcommand{\hecf@osuppp}{^{\smash{\lower-0.2ex\hbox{$\scriptstyle (o)$}}}}
\newcommand{\hecf@vsupp}{^{\smash{\lower1.0ex\hbox{$\scriptstyle \mathrm{vol}$}}}}
\newcommand{\hecf@ecsupp}{^{\smash{\lower1.3ex\hbox{$\scriptstyle \mathrm{EC}$}}}}

\newcommand{\hecf@ref}{\mathrm{ref}}

\newcommand{\hecf@U}{u}
\newcommand{\hecf@hatU}{U}

\newcommand{\hecf@metric}[4]{{#1\mkern-#2mu_{#3}#4}}
\newcommand{\hecf@tmetric}[4]{{\widetilde{#1}\mkern-#2mu_{#3}#4}}

\newcommand{\hecf@Hop}[3]{#1_{H#2}#3{}}
\newcommand{\hecf@tHop}[3]{{\widetilde #1}_{H#2}#3{}}

\newcommand{\hecf@HGop}[2]{#1_{H(G)}#2{}}
\newcommand{\hecf@tHGop}[2]{{\widetilde #1}_{H(G)}#2{}}

\newcommand{\hecf@Qnum}[2]{{\widetilde{\mathcal{Q}}}#1_{#2}{}}

\newcommand{\hecf@dotM}{\dot M}
\newcommand{\hecf@dottM}{\dot{\widetilde M}}

\newcommand{\pcutsub}{\hecf@pcutsub}
\newcommand{\ksup}{\hecf@ksup}
\newcommand{\osup}{\hecf@osup}
\newcommand{\osupp}{\hecf@osupp}
\newcommand{\osuppp}{\hecf@osuppp}
\newcommand{\vsupp}{\hecf@vsupp}
\newcommand{\ecsupp}{\hecf@ecsupp}

\newcommand{\U}{\hecf@U}
\newcommand{\Up}{\hecf@U_{\pcutsub}}

\newcommand{\UHAT}{\hecf@hatU}
\newcommand{\UHATO}{\hecf@hatU^{(o)}}
\newcommand{\UHATOref}{\hecf@hatU^{(o),\hecf@ref}}

\newcommand{\UHATp}{\hecf@hatU_{\pcutsub}{}}
\newcommand{\UHATOp}{\hecf@hatU_{\pcutsub}^{(o)}{}}
\newcommand{\UHATOpref}{\hecf@hatU_{\pcutsub}^{(o),\hecf@ref}{}}

\newcommand{\dUHAT}{\dot{\UHAT}}
\newcommand{\dUHATO}{\dot{\UHAT}\osup}

\newcommand{\dUHATp}{\dot{\UHAT}_{\pcutsub}{}}
\newcommand{\dUHATOp}{\dot{\UHAT}_{\pcutsub}\osup{}}

\newcommand{\HMETMAT}[2]{\hecf@metric{#1}{#2}{H}{}}
\newcommand{\HMETMATD}[2]{\hecf@metric{#1}{#2}{H,\Delta}{}}
\newcommand{\HMETMATO}[2]{\hecf@metric{#1}{#2}{H}{^{\;(o)}}}
\newcommand{\HMETMATOD}[2]{\hecf@metric{#1}{#2}{H,\Delta}{^{\;(o)}}}

\newcommand{\HMETMATAPP}[2]{\hecf@tmetric{#1}{#2}{H}{}}
\newcommand{\HMETMATDAPP}[2]{\hecf@tmetric{#1}{#2}{H,\Delta}{}}
\newcommand{\HMETMATOAPP}[2]{\hecf@tmetric{#1}{#2}{H}{^{\;(o)}}}
\newcommand{\HMETMATODAPP}[2]{\hecf@tmetric{#1}{#2}{H,\Delta}{^{\;(o)}}}

\newcommand{\HGMETMAT}[2]{\hecf@metric{#1}{#2}{H(G)}{}}
\newcommand{\HGMETMATO}[2]{\hecf@metric{#1}{#2}{H(G)}{^{\;(o)}}}
\newcommand{\HMETMATOG}[2]{\hecf@metric{#1}{#2}{H(G)}{^{\;(o)}}}
\newcommand{\HMETMATOGAPP}[2]{\hecf@tmetric{#1}{#2}{H(G)}{^{\;(o)}}}

\newcommand{\MH}{\hecf@Hop{M}{}{}}
\newcommand{\MHref}{\hecf@Hop{M}{}{^{\hecf@ref}}}
\newcommand{\MHAPP}{\hecf@tHop{M}{}{}}

\newcommand{\dMH}{{\hecf@dotM}_H{}}
\newcommand{\dMHref}{{\hecf@dotM}_H^{\hecf@ref}{}}
\newcommand{\dMHO}{{\hecf@dotM}^{(o)}_H{}}
\newcommand{\dMHV}{{\hecf@dotM}^{\mathrm{vol}}_H{}}

\newcommand{\dMHAPP}{{\hecf@dottM}_H{}}
\newcommand{\dMHOAPP}{{\hecf@dottM}\osupp_H{}}
\newcommand{\dMHVAPP}{{\hecf@dottM}\vsupp_H{}}

\newcommand{\dMHECCAPP}{{\hecf@dottM}\ecsupp_H{}}

\newcommand{\NH}{\hecf@Hop{N}{}{}}
\newcommand{\VH}{\hecf@Hop{V}{}{}}
\newcommand{\BH}{\hecf@Hop{B}{}{}}

\newcommand{\NHk}{\hecf@Hop{N}{}{\ksup}}
\newcommand{\VHk}{\hecf@Hop{V}{}{\ksup}}
\newcommand{\BHk}{\hecf@Hop{B}{}{\ksup}}

\newcommand{\NHX}{\hecf@Hop{N}{,X}{}}
\newcommand{\VHX}{\hecf@Hop{V}{,X}{}}
\newcommand{\BHX}{\hecf@Hop{B}{,X}{}}

\newcommand{\NHXk}{\hecf@Hop{N}{,X}{\ksup}}
\newcommand{\VHXk}{\hecf@Hop{V}{,X}{\ksup}}
\newcommand{\BHXk}{\hecf@Hop{B}{,X}{\ksup}}

\newcommand{\NHA}{\hecf@Hop{N}{,A}{}}
\newcommand{\VHA}{\hecf@Hop{V}{,A}{}}
\newcommand{\BHA}{\hecf@Hop{B}{,A}{}}

\newcommand{\NHD}{\hecf@Hop{N}{,\Delta}{}}
\newcommand{\VHD}{\hecf@Hop{V}{,\Delta}{}}
\newcommand{\BHD}{\hecf@Hop{B}{,\Delta}{}}

\newcommand{\NHAPP}{\hecf@tHop{N}{}{}}
\newcommand{\VHAPP}{\hecf@tHop{V}{}{}}
\newcommand{\BHAPP}{\hecf@tHop{B}{}{}}

\newcommand{\NHAPPK}{\hecf@tHop{N}{}{\ksup}}
\newcommand{\VHAPPK}{\hecf@tHop{V}{}{\ksup}}
\newcommand{\BHAPPK}{\hecf@tHop{B}{}{\ksup}}

\newcommand{\NHXAPP}{\hecf@tHop{N}{,X}{}}
\newcommand{\VHXAPP}{\hecf@tHop{V}{,X}{}}
\newcommand{\BHXAPP}{\hecf@tHop{B}{,X}{}}

\newcommand{\NHXAPPk}{\hecf@tHop{N}{,X}{\ksup}}
\newcommand{\VHXAPPk}{\hecf@tHop{V}{,X}{\ksup}}
\newcommand{\BHXAPPk}{\hecf@tHop{B}{,X}{\ksup}}

\newcommand{\NHAAPP}{\hecf@tHop{N}{,A}{}}
\newcommand{\VHAAPP}{\hecf@tHop{V}{,A}{}}
\newcommand{\BHAAPP}{\hecf@tHop{B}{,A}{}}

\newcommand{\NHDAPP}{\hecf@tHop{N}{,\Delta}{}}
\newcommand{\VHDAPP}{\hecf@tHop{V}{,\Delta}{}}
\newcommand{\BHDAPP}{\hecf@tHop{B}{,\Delta}{}}

\newcommand{\NHDAAPP}{\hecf@tHop{N}{,\Delta A}{}}
\newcommand{\VHDAAPP}{\hecf@tHop{V}{,\Delta A}{}}
\newcommand{\BHDAAPP}{\hecf@tHop{B}{,\Delta A}{}}

\newcommand{\NHO}{\hecf@Hop{N}{}{^{(o)}}}
\newcommand{\VHO}{\hecf@Hop{V}{}{^{(o)}}}
\newcommand{\BHO}{\hecf@Hop{B}{}{^{(o)}}}

\newcommand{\NHOD}{\hecf@Hop{N}{,\Delta}{^{(o)}}}
\newcommand{\VHOD}{\hecf@Hop{V}{,\Delta}{^{(o)}}}
\newcommand{\BHOD}{\hecf@Hop{B}{,\Delta}{^{(o)}}}

\newcommand{\NHOSKW}{N_{H\;\mathrm{skw}}^{(o)}{}}
\newcommand{\BHOG}{\hecf@HGop{B}{^{(o)}}}

\newcommand{\NHOAPP}{\hecf@tHop{N}{}{^{(o)}}}
\newcommand{\VHOAPP}{\hecf@tHop{V}{}{^{(o)}}}
\newcommand{\BHOAPP}{\hecf@tHop{B}{}{^{(o)}}}

\newcommand{\BHODAPP}{\hecf@tHop{B}{,\Delta}{^{(o)}}}
\newcommand{\NHOAPPSKW}{{\widetilde N}_{H\;\mathrm{skw}}^{(o)}{}}
\newcommand{\BHOGAPP}{\hecf@tHGop{B}{^{(o)}}}

\newcommand{\QHNUM}{\hecf@Qnum{}{H}}
\newcommand{\QHNUMO}{\hecf@Qnum{^{(o)}}{H}}
\newcommand{\QHGNUM}{\hecf@Qnum{}{H(G)}}
\newcommand{\QHGNUMO}{\hecf@Qnum{^{(o)}}{H(G)}}

\makeatother

\begin{document}
\maketitle


\begin{abstract}
  This paper derives a design principle for structure-preserving
  Galerkin formulations of energy-conserving hyperbolic systems. The
  aim is to reproduce the modal-energy-exchange structure of the
  continuous system within a resolved finite-mode space. Total energy
  conservation follows from this structure. We introduce a
  state-dependent physical-energy metric \(H\) and derive the
  corresponding energy-compatibility identity. In the infinite-mode
  exact-integration model, the volume contribution has an
  antisymmetric representation after \(H\)-orthogonalization, yielding
  pairwise modal energy exchange. Interface contributions take the
  same exchange form. To reproduce this structure in the practical
  finite-mode system, we combine two constructions: a Galerkin
  projection coupled with the physical-energy metric that guarantees
  the \(H\)-metric summation-by-parts identity, and an
  energy-compatibility closure that removes the component of the
  compatibility action contributing to the scalar energy residual.
  With energy-compatible one-sided boundary actions at interfaces,
  they close the total-energy balance of the finite-mode system and
  recover an explicit antisymmetric interelement pairwise exchange
  consistent with the exact-integration reference. We also compare the practical
  operator construction with the finite-mode exact-integration
  reference and obtain an \(O(h^{p+1})\) defect estimate. Finally, we
  derive an equivalent form of the resulting equation in the fixed
  Galerkin basis for direct implementation.
\end{abstract}



\setlength{\emergencystretch}{2.0em}
\section{Introduction}
\label{sec:intro}

When a partial differential equation (PDE) is discretized, its
dynamics split into resolved and unresolved components. This raises a
fundamental question: what should be regarded as the correct resolved
dynamics? The answer depends on which characteristics of the PDE the
resolved dynamics inherit. A discrete system is often regarded as an
approximation of the original PDE. This paper takes the position that,
from an energy-structure perspective, it should also reproduce the
energy structure of the PDE.
This perspective requires identifying the energy to be represented and
introducing a metric that quantifies it. This paper focuses on physical energy.
We view the resolved--unresolved split as a truncation of a modal
representation of the continuous dynamics. Under this metric, the
modal dynamics describe how physical energy is exchanged among modes.
The essential requirement is that the resolved dynamics retain the
modal-energy-exchange structure induced by the continuous system
without introducing spurious energy transfer, artificial dissipation,
or non-physical interaction pathways. This principle follows
Arakawa's philosophy of structural preservation
\cite{Arakawa1966,ArakawaLamb1981}. It is also essential to the closure
problem for unresolved dynamics, because unresolved effects should be
modeled on a resolved system that already has the correct energy
structure. Otherwise, we cannot consistently determine how unresolved
effects should act on the discrete system.

A Galerkin formulation is useful for representing modal dynamics. We
consider three levels in the model hierarchy: the continuous
system, the finite-mode exact-integration reference system whose
dynamics evolve in a finite-mode space with all weak-form integrals
evaluated exactly, and the practical finite-quadrature system that
implements the finite-mode dynamics using finite-degree quadrature.
Our design requirement is that all three systems exhibit the same
modal-energy-exchange structure and conserve the total physical
energy.

This hierarchy separates truncation and quadrature defects. The former
arises from removing unresolved components between the continuous and
finite-mode exact-integration reference systems, whereas the latter
arises between the finite-mode exact-integration reference and
practical finite-quadrature systems and includes numerical and
structural errors from finite-degree quadrature. The intermediate
reference separates the defects, provides a benchmark for resolved
dynamics, and is the target approached by the practical system as the
quadrature degree increases. However, finite-degree quadrature does
not generally preserve the energy structure of the continuous system.
This paper presents a design principle for recovering this structure
in the practical system.

We consider an energy-conserving hyperbolic system in \(d\) spatial
dimensions:
\begin{equation}
  \partial_t \U + A_k(\U)\,\partial_{x_k}\U = 0,
  \label{eq:ql-system}
\end{equation}
where \(x\in\mathbb{R}^d\) is the spatial coordinate,
\(\U\in\mathbb{R}^m\) is the state variable with \(m\) components, and
\(A_k(\U)\) is the flux Jacobian. We use the Discontinuous Galerkin
(DG) method as the Galerkin realization. Its element-local structure is
well suited to massively parallel computation
\cite{HesthavenWarburton2008}, and DG belongs to the mature class of
high-order element-based methods developed together with
spectral-element formulations \cite{Karniadakis2005}.

A basic tool for stable high-order DG methods is the
summation-by-parts (SBP) structure, which mimics integration by parts
at the discrete level. Originally introduced for the stability
analysis of finite difference methods for hyperbolic equations, SBP
was developed into a systematic framework for constructing stable
high-order difference operators
\cite{KreissScherer1974,Strand1994}. Combined with the simultaneous
approximation term (SAT) for weakly imposing boundary and
interelement conditions, SBP--SAT provides a standard framework for
stable high-order discretizations of time-dependent problems
\cite{CarpenterGottliebAbarbanel1994,Carpenter1999}. Its relation to
DG and discontinuous Galerkin spectral element methods was clarified
within this framework \cite{Gassner2013}.
For nonlinear problems, split forms suppress aliasing and
non-physical energy generation by exploiting algebraically equivalent
representations of nonlinear terms to satisfy a discrete
convex-entropy compatibility relation
\cite{CarpenterFisherNordstrom2014,Gassner2016}. Flux differencing
realizes such forms in physical space. By combining Tadmor's
two-point entropy-conservative flux \cite{Tadmor1987} with the SBP
structure, it produces a finite-volume-like cancellation within each
element
\cite{FisherCarpenter2013,CarpenterFisherNielsenFrankel2014}.
This construction underlies many modern entropy-stable DG methods
\cite{Gassner2016,ChenShu2017JCP}. The entropy-conservative volume
discretization prevents entropy generation within each element, while
an entropy-stable SAT-based interface flux supplies the required
dissipation
\cite{FisherCarpenter2013,CarpenterFisherNielsenFrankel2014,ChenShu2017JCP}.
Thus, SBP--SAT, split forms, flux differencing, and Tadmor-type
two-point fluxes form a central framework for modern high-order
entropy-stable DG methods \cite{Ranocha2023}.

The design theory in this paper is mathematically based on
convex-entropy theory, but its design objective differs from that of
existing entropy-conservative/stable methods. First, the target is
physical energy rather than a convex entropy; in general, a convex
entropy is not identical to the physical energy of the system.
Second, we design the modal dynamics of the physical energy to mimic
those of the continuous system. Existing entropy-conservative/stable
DG formulations construct elementwise entropy or energy balances
through node-pair cancellations in physical space and interface
fluxes. These structures depend on the nodal locations and quadrature
weights and are not designed to ensure faithful modal energy
exchange. Thus, even when transformed into modal form, the resulting
operator need not preserve the modal-energy-exchange structure of the
continuous system.

To realize a closure based on modal energy exchange, we need a metric
that represents the physical energy itself. This metric should not be
frozen in space or time, because its state dependence enters the modal
energy exchange. We introduce a symmetric positive definite
(SPD) matrix field \(H(\U)\in\mathbb{R}^{m\times m}\) and represent the
physical energy density by
\begin{equation}
  e(\U):=\tfrac12 \U^\top H(\U)\U.
  \label{eq:h-energy-density}
\end{equation}
This paper shows how to close the resolved finite-mode dynamics under
this metric so that the discrete system faithfully reproduces modal
energy exchange and conserves the total physical energy.
The resulting formulation is also implementable in finite-mode
simulations while preserving the modal-energy-exchange structure.

\begin{figure}[t]
  \centering
  \includegraphics[width=1.0\linewidth]{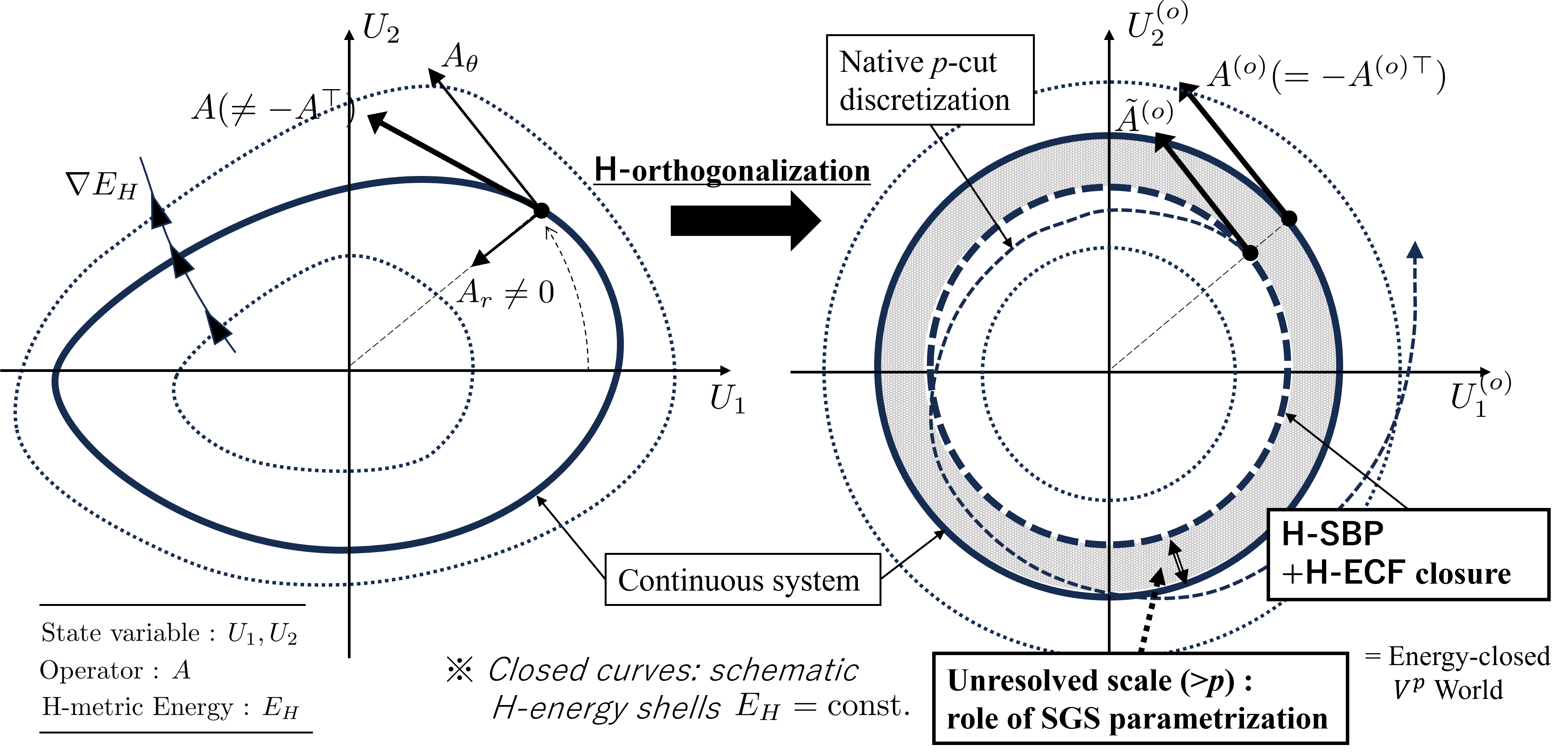}
  \caption{ Schematic phase-space picture of the model hierarchy.  The
    dotted closed curves denote constant physical-energy surfaces.
    The thick solid curve represents the continuous dynamics, the thin
    dashed curve the native \(p\)-cut dynamics, and the thick dashed
    curve the finite-mode dynamics reconstructed by the \(H\)--SBP
    identity and \(H\)--ECF. The left panel shows the original
    coordinates, and the right panel shows the \(H\)-orthogonalized
    coordinates, where the physical energy becomes one-half the
    squared Euclidean norm.  }
  \label{fig:model-roadmap}
\end{figure}

\Cref{fig:model-roadmap} summarizes the viewpoint and contribution of
this paper.
After \(H\)-orthogonalization, the physical energy becomes one-half
the squared Euclidean norm, and the continuous volume contribution can
be written in antisymmetric form as pairwise modal exchange on an
energy shell.
A native \(p\)-cut does not generally preserve this structure.
We recover it in \(V^p\) by enforcing the \(H\)--SBP and
energy-compatibility identities. The latter is achieved by
\(H\)--Energy Compatibility Forcing (\(H\)--ECF), as explained in
\Cref{sec:h-ecf}.
The shaded region represents the unresolved components removed by
modal truncation and hence the target of subgrid-scale (SGS) modeling.

In \Cref{sec:h-foundation}, we construct the state-dependent physical
energy metric \(H(\U)\) by normalizing the physical-energy Hessian so
that \cref{eq:h-energy-density} holds exactly, and derive the
corresponding energy-compatibility identity. After introducing the
discrete notation and assumptions in \Cref{sec:notation}, the main
contributions are as follows.
\begin{enumerate}
  [label=(\roman*), leftmargin=1.5em, labelsep=0.4em, itemsep=0.3em,
    topsep=0.3em ]
\item
  We construct the infinite-mode exact-integration Galerkin (EIG)
  system, denoted I--EIG, and show that, after
  \(H\)-orthogonalization, the volume contribution can be written in
  antisymmetric form as pairwise modal exchange within each element.
  Interface contributions take the same pairwise form and cancel
  across internal interfaces, leaving only external-boundary
  contributions
  (\Cref{sec:i-eig}).
\item
  For the practical finite-quadrature system, we derive structural
  conditions for the \(H\)--SBP identity and introduce Coupled
  Projection Galerkin (\(\mathrm{CPG}\)), which consistently projects
  coupled \(H\)-weighted fields. CPG guarantees the \(H\)--SBP identity
  but does not alone close energy compatibility
  (\Cref{sec:discrete-defect}).
\item
  We introduce \(H\)--ECF, which removes the component of the
  compatibility action that contributes to the scalar energy residual
  and represents the remaining tangential action by an antisymmetric
  lift. Combined with CPG, it recovers pairwise modal energy exchange
  in the \(H\)-orthogonalized representation. We also derive the
  equivalent fixed-basis velocity equation, which is linear in the
  time derivative of the fixed-basis coefficients and provides an
  implementation of the \(H\)--ECF/CPG system
  (\Cref{sec:h-ecf}).
\item
  We compare the practical \(H\)--ECF/CPG construction with its
  finite-mode exact-integration counterpart \(H\)--ECF/EIG and
  establish an \(O(h^{p+1})\) estimate for the defect in the
  antisymmetric modal-energy-exchange operator. The estimate separates
  the high-order projection contribution from the quadrature-error
  contribution in the operator-construction layer
  (\Cref{sec:exchange-defect}).
\end{enumerate}
Finally, \Cref{sec:discussion} summarizes the model hierarchy and
interprets the exchange-defect estimate.

\section{Physical energy metric and energy-compatibility identity}
\label{sec:h-foundation}

Convex entropy theory provides the algebraic construction used here:
an entropy pair satisfies a compatibility condition, and the
Hessian of the energy symmetrizes the flux Jacobian
\cite{FriedrichsLax1971,Dafermos2016}.  We apply this
route to systems and state variables for which the physical energy
itself is a nondegenerate convex entropy.  After a scalar
normalization of the Hessian, we obtain an SPD metric \(H(\U)\) that
represents the physical energy exactly as the quadratic form in
\cref{eq:h-energy-density}.  This section then derives the
\(H\)-metric energy-compatibility identity used throughout this paper.
The next two lemmas recall standard consequences of convex entropy
theory.
\begin{lemma}[Physical energy-pair compatibility]
  \label{lem:energy-pair}
  Let \(\U\) be a smooth solution of \cref{eq:ql-system}.  Let
  \(e(\U)\) and \(g_k(\U)\) be a scalar physical energy density and the
  corresponding scalar energy fluxes, respectively.  If
  \begin{equation}
    \nabla_\U e(\U)^T A_k(\U) = \nabla_\U g_k(\U)^T ,
    \label{eq:energy-pair}
  \end{equation}
  then \(e\) and \(g_k\) satisfy the local energy conservation law
  \begin{equation}
    \partial_t e(\U) + \partial_{x_k} g_k(\U) = 0
    \label{eq:energy-conservation}
  \end{equation}
  along solutions of \cref{eq:ql-system}.  Conversely, if
  \cref{eq:energy-conservation} holds for all
  solutions of \cref{eq:ql-system}, then
  \cref{eq:energy-pair} holds.
\end{lemma}
\begin{proof}
  The chain rule and \cref{eq:ql-system} give
  \( \partial_t e(\U) + \partial_{x_k}g_k(\U)
  = [ -\nabla_\U e(\U)^T A_k(\U) + \nabla_\U g_k(\U)^T ]
  \partial_{x_k}\U \).
  Hence, \cref{eq:energy-conservation} holds if and only if
  \cref{eq:energy-pair} holds.
\end{proof}

\begin{lemma}[Hessian-based symmetrizer]
  \label{lem:hessian-sym}
  Assume that \(A_k(\U)\) is the Jacobian of a conservative state flux,
  \(A_k(\U)=\partial_\U F_k(\U)\), and that the physical energy pair
  satisfies \cref{eq:energy-pair}.  Define the Hessian as
  \( \mathcal{H}_e(\U) := \nabla_\U^2 e(\U) \).
  Then
  \begin{equation}
    \mathcal{H}_e(\U)A_k(\U)=[\mathcal{H}_e(\U)A_k(\U)]^{\top} .
    \label{eq:hessian-sym}
  \end{equation}
\end{lemma}
\begin{proof}
  In components, \cref{eq:energy-pair} is
  \(e_iA_{k,ij}=(g_k)_j\).  Differentiating with respect to \(\U_\ell\)
  and antisymmetrizing in \(j\) and \(\ell\) cancels the second
  derivatives of \(g_k\).  The differentiated \(A_k\)-terms also cancel
  because \(A_{k,ij}=\partial_{\U_j}F_{k,i}\).  Thus,
  \(e_{i\ell}A_{k,ij}=e_{ij}A_{k,i\ell}\), which is exactly
  \cref{eq:hessian-sym}.
\end{proof}

The Hessian \(\mathcal H_e(\U)\) is the natural perturbation-energy
metric: near \(\U_0\), the quadratic part of \(e(\U_0+\delta\U)\) is
\(\tfrac12\delta\U^\top\mathcal H_e(\U_0)\delta\U\).  In general,
however, \(e(\U)\ne\tfrac12\U^\top\mathcal H_e(\U)\U\).  Therefore, we
normalize the Hessian by a positive scalar so that the physical energy is
represented exactly by the resulting quadratic form.
\begin{definition}[Scalar-normalized physical energy metric]
  \label{def:h-metric}
  Assume that \(\mathcal H_e(\U)\) is SPD and \(e(\U)>0\) on the
  admissible state set.  Define
  \( q_{\mathcal H}(\U) := \tfrac12 \U^\top \mathcal H_e(\U)\U \).
  The physical energy metric is
  \begin{equation}
    H(\U)
    :=
    \frac{e(\U)}{q_{\mathcal H}(\U)}\,\mathcal H_e(\U)
    =
    \frac{2e(\U)}{\U^\top \mathcal H_e(\U)\U}\,\mathcal H_e(\U).
    \label{eq:h-metric}
  \end{equation}
  Then \(H(\U)\) is SPD as a positive scalar multiple of
  \(\mathcal{H}_e(\U)\) and satisfies \cref{eq:h-energy-density}.
\end{definition}

\begin{lemma}[\(H\)-symmetrization of the flux Jacobian]
  \label{lem:h-metric-sym}
  Let \(H(\U)\) be defined by
  \cref{eq:h-metric}. Then \(H(\U)\)
  symmetrizes the flux Jacobian:
  \begin{equation}
    H(\U)A_k(\U)=[H(\U)A_k(\U)]^{\top} .
    \label{eq:h-sym}
  \end{equation}
\end{lemma}
\begin{proof}
  Since \(H(\U)\) is a positive scalar multiple of \(\mathcal{H}_e(\U)\),
  \cref{eq:hessian-sym} implies
  \cref{eq:h-sym}.
\end{proof}

\begin{table}[t]
  \centering
  \caption{Examples in which the physical energy is a nondegenerate
    convex entropy in the chosen state variables.  The last column
    lists the scalar normalization \(\lambda(\U)\).  Here, \(u\) and
    \(\mathbf u\) denote velocities, \(\mathbf m\) momentum, \(h\)
    depth, \(g\) gravitational acceleration, \(\rho\) density,
    \(\eta\) entropy density, and \(\kappa\) the ratio of specific
    heats.
    For the Euler equations, we set \( s=\frac{\eta}{\rho},
    p=\rho^\kappa \exp(s), \theta=\frac{p}{(\kappa-1)\rho^2}, \) and
    \(a=s-\kappa+1.\)}
  \label{tab:h-metrics}
  \resizebox{0.96\linewidth}{!}{%
    \begin{tabular}{@{}llll@{}}
      \toprule
      \begin{tabular}{l}
        System and \\ state variables
      \end{tabular}
      & \(e(\U)\) & \(\mathcal{H}_e(\U)=\nabla_\U^2 e(\U)\) &
      \(\lambda(\U):=\frac{2e(\U)}{\U^\top\mathcal{H}_e(\U)\U}\)
      \\ \midrule

      \begin{tabular}{l}
        Shallow water eq.\\ \(\U=(h,\mathbf m)\)
      \end{tabular}
      & \(\displaystyle \frac{|\mathbf m|^2}{2h} + \tfrac12 gh^2 \) &
      \(\displaystyle
      \begin{pmatrix}
        \dfrac{|\mathbf m|^2}{h^3}+g & -\dfrac{\mathbf m^{\top}}{h^2}
        \\[0.5em] -\dfrac{\mathbf m}{h^2} & \dfrac1h I_2
      \end{pmatrix}
      \) & \(\displaystyle 1+\frac{|\mathbf m|^2}{gh^3} \) \\[2.2em]

      \begin{tabular}{l}
        Ideal-gas Euler eq.\\ \(\U=(\rho,\mathbf m,\eta)\)
      \end{tabular}
      & \(\displaystyle \frac{|\mathbf m|^2}{2\rho} +
      \frac{p}{\kappa-1} \) & \(\displaystyle
      \begin{pmatrix}
        \dfrac{|\mathbf m|^2}{\rho^3} + \theta(a^2+\kappa-1) &
        -\dfrac{\mathbf m^{\top}}{\rho^2} & -\theta a \\[0.5em]
        -\dfrac{\mathbf m}{\rho^2} & \dfrac1\rho I_3 & \mathbf 0
        \\[0.5em] -\theta a & \mathbf 0^{\top} & \theta
      \end{pmatrix}
      \) & \(\displaystyle \frac{|\mathbf m|^2}
         {\kappa\rho^{\kappa+1}\exp(\eta/\rho)}
         +\frac{2}{\kappa(\kappa-1)} \) \\[3.0em]

           \bottomrule
    \end{tabular}%
  }
\end{table}

The construction depends on the choice of state variables. For the
ideal-gas Euler equations in conservative variables
\((\rho,\mathbf m,e)\), where \(\rho\) is the mass density and
\(\mathbf m\) is the momentum density, the Hessian of the physical
energy is degenerate because \(e\) itself is a state variable. In the
variables \((\rho,\mathbf m,\eta)\), where \(\eta\) is the entropy
density, the Hessian becomes nondegenerate. By
\Cref{def:h-metric}, the corresponding \(H(\U)\) is then SPD.
\Cref{tab:h-metrics} lists representative systems and state variables
for which this holds.

To derive the \(H\)-metric energy-compatibility identity, represent
each energy-flux component \(g_k(\U)\) as
\begin{equation}
  g_k(\U)=\tfrac12 \U^{\top} G_k(\U)\U ,
  \label{eq:energy-flux-form}
\end{equation}
where \(G_k(\U)\) is symmetric because the antisymmetric part does not
contribute to \(\U^{\top}G_k(\U)\U\).

\begin{lemma}[\(H\)-symmetry of the flux defect]
  \label{lem:flux-defect-h-sym}
  Let \(G_k(\U)\) be the symmetric matrix in
  \cref{eq:energy-flux-form}.  Define \(\Delta A_k(\U)\) by
  \( G_k(\U)=H(\U) [ A_k(\U)+\Delta A_k(\U) ] \).
  Then
  \begin{equation}
    H(\U)\Delta A_k(\U) = [H(\U)\Delta A_k(\U)]^{\top}.
    \label{eq:h-sym2}
  \end{equation}
\end{lemma}
\begin{proof}
  Symmetry of \(G_k\) and \cref{eq:h-sym} gives \cref{eq:h-sym2}.
\end{proof}
Thus, \(H(\U)\) symmetrizes both \(A_k(\U)\) and \(\Delta A_k(\U)\).

\begin{proposition}[\(H\)-metric energy-compatibility identity]
  \label{prop:h-compat}
  Assume that the physical energy pair satisfies
  \cref{eq:energy-pair}.  Let \(H(\U)\) be defined by
  \Cref{def:h-metric}. Let \(G_k(\U)\) and
  \(\Delta A_k(\U)\) be as in \Cref{lem:flux-defect-h-sym}.
  Then the solution of \cref{eq:ql-system} satisfies
  \begin{equation}
    \U^{\top} H\Delta A_k\,\partial_{x_k}\U
    + \tfrac12 \U^{\top}
    \left[
      \partial_t H
      + \partial_{x_k}\{H(A_k+\Delta A_k)\}
    \right]\U
    =0 .
    \label{eq:h-compat}
  \end{equation}
\end{proposition}
\begin{proof}
  Expanding \cref{eq:energy-conservation}
  using \Cref{def:h-metric},
  \cref{eq:energy-flux-form}, and
  \cref{eq:ql-system} gives
  \(
    \U^{\top}(G_k-HA_k)\partial_{x_k}\U
    +\tfrac12 \U^{\top}[\partial_tH+\partial_{x_k}G_k]\U=0
  \).
  \(G_k(\U)=H(\U) [A_k(\U)+\Delta A_k(\U)]\)
  gives \cref{eq:h-compat}.
\end{proof}
This identity provides the continuous energy-compatibility identity,
Using it, we derive its Galerkin counterpart
after introducing notations and assumptions.

\section{Notation and assumptions in the Galerkin formulation}
\label{sec:notation}

This section introduces the Galerkin notation and assumptions used
in the subsequent sections.

\begin{definition}[Mode expansion and coefficient ordering]
  \label{def:vp-expansion}
  Let \(\U\) be the state variable in \cref{eq:ql-system}.  On an
  element \(\Omega\), let \(V^p:=V^p(\Omega)\) be the polynomial space of
  degree at most \(p\), and set \(N_p:=\dim V^p\).  Let
  \(I_m\in\mathbb R^{m\times m}\) and
  \(I:=I_{mN_p}\in\mathbb R^{mN_p\times mN_p}\).
  Let \(\{\phi_\alpha\}_{\alpha=1}^{N_p}\) be a time-independent
  \(L^2\)-orthonormal basis of \(V^p\), and define
  \(
    \Phi_\alpha(x):=\phi_\alpha(x)I_m\in\mathbb R^{m\times m},
    \Phi(x):=[\Phi_1(x)\ \cdots\ \Phi_{N_p}(x)]
    \in\mathbb R^{m\times mN_p}
  \).
  The \(p\)-mode truncation \(\Up\in(V^p)^m\) is written as
  \[
  \Up(x,t)=\sum_{\alpha=1}^{N_p}\Phi_\alpha(x)\,\UHATp_\alpha(t),
  \qquad \UHATp_\alpha(t)\in\mathbb{R}^m.
  \]
  Its coefficient vector is
  \(
    \UHATp(t):=
    [\UHATp_1(t)\ \cdots\ \UHATp_{N_p}(t)]^\top
    \in\mathbb R^{mN_p},
  \)
  ordered mode-major, with the \(m\) physical components grouped within
  each mode.
\end{definition}
\begin{definition}[Block transpose and symmetric/skew parts]
  \label{def:block-sym-skw}
  In mode-major order, matrices are viewed as \(N_p\times N_p\) block
  matrices with \(m\times m\) blocks.  For
  \(X=[X_{\alpha\beta}]_{\alpha,\beta=1}^{N_p}\), define
  \( (X^\top)_{\alpha\beta}:=(X_{\beta\alpha})^\top\),
  \( X_{\rm sym}:=\tfrac12(X+X^\top)\), and
  \(  X_{\rm skw}:=\tfrac12(X-X^\top) \).
\end{definition}
\begin{definition}[Discrete quadrature rules]
  \label{def:quadrature}
  For a scalar integrand \(g\), let \(Q_r[g]\) denote an \(r\)-exact
  volume quadrature with positive weights, and let
  \(Q_r^{\partial\Omega}[g]\) denote an \(r\)-exact face quadrature
  on \(\partial\Omega\) with positive weights.  For vector- or
  matrix-valued integrands, the same notation is applied componentwise.
\end{definition}
\begin{definition}[$L^2$ projections]
  \label{def:l2-h-proj}
  For an \(m\)-component function \(f\), the \(L^2\)-projection
  \(\Pi_{L^2,p}f\in(V^p)^m\) is defined by
  \(
    \int_\Omega \Phi^\top
    \bigl(f-\Pi_{L^2,p}f\bigr)\,dV = \mathbf 0
  \).
  If a quadrature rule \(Q\) is used, the corresponding discrete
  projection is denoted by \(\Pi_{L^2,p}^{\,Q}\).  For matrix-valued
  functions, both projections are applied componentwise.
\end{definition}

The error estimates are restricted to smooth regimes; nonsmooth regimes
requiring shock-capturing stabilization are outside the scope of the
analysis.
\begin{assumption}[Shape regularity]
  \label{ass:shape-regular}
  The element \(\Omega\) is convex and satisfies
  \(h/\rho\le C_{\rm sr}\), where \(h\) is its diameter, \(\rho\) is its
  inradius, and \(C_{\rm sr}>0\) is independent of \(h\).
\end{assumption}
\begin{assumption}[Admissibility and uniform bounds for the \(H\)-metric]
  \label{ass:h-bounds}
  All states under consideration take values in an admissible set
  \(\mathcal D\subset\mathbb R^m\).  On \(\mathcal D\), all eigenvalues
  of \(H(\U)\) lie in \([c_H,C_H]\) for constants
  \(0<c_H\le C_H<\infty\).
\end{assumption}
\begin{assumption}[Uniform bounds for \(\Up\) and its derivatives]
  \label{ass:up-bounds}
  The resolved state \(\Up\in(V^p)^m\) satisfies
  \(\|\Up\|_{L^\infty(\Omega)}\le C_0\),
  \( \|\nabla\Up\|_{L^\infty(\Omega)}\le C_1\),
  \( \|\partial_t\Up\|_{L^\infty(\Omega)}\le C_2 \)
  with constants independent of \(h\).
\end{assumption}
\begin{assumption}[Projection-error bounds for generated matrix fields]
  \label{ass:field-proj-error}
  The matrix fields generated from the resolved state \(\Up\) and the
  fixed coefficient direction \(V\), including
  \( H(\Up) \),\( H(\Up)A_k(\Up)\),\( H(\Up)\Delta A_k(\Up)\), and
  \(  D_{\Up} H(\Up)[\Phi V]
  \) are assumed to be well resolved up to degree \(p+1\) on each
  element.
\end{assumption}

%
\section{Reference \(H\)-metric energy structure}
\label{sec:i-eig}
This section clarifies the energy-exchange structure of the
infinite-mode exact-integration Galerkin system, which corresponds to
the continuous system.
The \(H\)-metric integration-by-parts and energy-compatibility identities
imply that, after \(H\)-orthogonalization, the volume contribution becomes
purely pairwise modal energy exchange.  Boundary contributions also cancel
pairwise across internal interfaces.
As a result, the energy balance closes on any target domain, with only
external boundary contributions remaining.

We first define the \(H\)-metric mass matrix and the associated operators.
\begin{definition}[$H$--metric operator notation]
  \label{def:h-operators}
  For a state field \(\U_\ast\), the \(H\)-metric mass matrix is
  \begin{align*}
    {\MH[\U_\ast]}_{\alpha\beta} &:= \textstyle{\int_\Omega} \Phi_\alpha^\top
    H(\U_\ast)\Phi_\beta\,dV .
  \end{align*}
  Let \(X_k(\U_\ast)\in\mathbb R^{m\times m}\) be a matrix-valued field
  depending on \(\U_\ast\).  The \(H\)-metric operators associated with
  \(X_k(\U_\ast)\) are  
  \begin{align*}
    \NHk[X(\U_\ast)]_{\alpha\beta}
    &:= \textstyle{\int_\Omega}
    \Phi_\alpha^\top H(\U_\ast)X_k(\U_\ast)
    \partial_{x_k}\Phi_\beta\,dV, \\
    \VHk[X(\U_\ast)]_{\alpha\beta}
    &:= \textstyle{\int_\Omega} \Phi_\alpha^\top
    \partial_{x_k}\!\left( H(\U_\ast)X_k(\U_\ast) \right)
    \Phi_\beta\,dV, \\
    \BHk[X(\U_\ast)]_{\alpha\beta}
    &:= \textstyle{\int_{\partial\Omega}} \Phi_\alpha^\top H(\U_\ast)X_k(\U_\ast)
    \Phi_\beta n_k\,dS .
  \end{align*}
  Here \(n=(n_1,\ldots,n_d)^\top\) denotes the outward unit normal
  on \(\partial\Omega\).  
  Set  \( \NH[X]:=\sum_k\NHk[X] \), \( \VH[X]:=\sum_k\VHk[X] \),
  and \( \BH[X]:=\sum_k\BHk[X]\).
\end{definition}

\subsection{Semi-discrete coefficient equation under the $H$--metric}
\label{subsec:h-weak-form}

We formulate the native \(p\)-Galerkin system and then take its
infinite-mode exact-integration limit as the reference system.
\begin{definition}[Native \(p\)-Galerkin system]
  \label{def:native-pg}
  Let \(\Up\in (V^p)^m\) be the state defined in
  \Cref{def:vp-expansion}.  By \Cref{def:h-operators},
  set \( \MH := \MH[\Up] \), \( \NH := \NH[A(\Up)] \),
  \( \VH := \VH[A(\Up)] \), and \( \BH := \BH[A(\Up)] \).
  Let \(\QHNUM_\alpha[\Up^-,\Up^+]\) be the boundary/interface numerical
  flux functional.
  Define
  \begin{equation}
    \mathcal R_p := \partial_t \Up + A_k(\Up)\partial_{x_k}\Up.
    \label{eq:residual-p}
  \end{equation}
  The \(H(\Up)\)-metric testing of \(\mathcal R_p\), with the boundary
  correction
  \(
    \QHNUM_\alpha[\Up^-,\Up^+]
    -\sum_{\beta=1}^{N_p}\BH_{\alpha\beta}\UHATp_\beta
  \),
  gives, for \(\alpha=1,\ldots,N_p\),
  \begin{equation}
    \sum_{\beta=1}^{N_p} \MH_{\alpha\beta}\dUHATp_\beta
    + \sum_{\beta=1}^{N_p} \NH_{\alpha\beta}\UHATp_\beta
    =
    \sum_{\beta=1}^{N_p} \BH_{\alpha\beta}\UHATp_\beta
    - \QHNUM_\alpha[\Up^-,\Up^+] .
    \label{eq:h-weak-pg}
  \end{equation}
\end{definition}
\begin{definition}[Infinite-mode exact-integration Galerkin (I--EIG) system]
  \label{def:i-eig}
  The I--EIG system is the \(p\to\infty\) limit of
  \cref{eq:h-weak-pg}, with \(\Up\) replaced by the state \(\U\).
\end{definition}

\begin{lemma}[Vanishing boundary correction in I--EIG]
  \label{lem:i-eig-no-trace}
  In the I--EIG system, the limiting boundary functional in
  \cref{eq:h-weak-pg} is consistent with the physical
  boundary trace and satisfies
  \(
    \QHNUM_\alpha =
    \sum_{\beta=1}^{\infty}\BH_{\alpha\beta}\UHAT_\beta 
  \).
  Consequently, the I--EIG coefficient equation has no right-hand
  side:
  \begin{equation}
    \sum_{\beta=1}^{\infty} \MH_{\alpha\beta}\dUHAT_\beta +
    \sum_{\beta=1}^{\infty} \NH_{\alpha\beta}\UHAT_\beta =0, \qquad
    \alpha\in\mathbb N .
    \label{eq:i-eig-ode}
  \end{equation}
\end{lemma}
\begin{proof}
  In the I--EIG limit, \(\Up\to\U\) and
  \(\partial_t\U+A_k(\U)\partial_{x_k}\U=0\) holds pointwise.  Thus, the
  boundary trace is single-valued across each internal face, and the
  boundary functional reduces to the physical boundary functional:
  \(
    \QHNUM_\alpha =
    \sum_{\beta=1}^{\infty}\BH_{\alpha\beta}\UHAT_\beta
  \).
  Substitution into the limiting form of \cref{eq:h-weak-pg} gives
  \cref{eq:i-eig-ode}.
\end{proof}

\subsection{Two identities in the I--EIG system}
The I--EIG system satisfies the following two identities.
\begin{lemma}[$H$--metric integration-by-parts ($H$--IBP) identity]
  \label{lem:h-ibp}
  For each direction \(k\), if \(H X_k=(H X_k)^\top\), then
  \(
  \NHk[X]+\NHk[X]^\top = \BHk[X]-\VHk[X]
  \).
  Summing over \(k\) gives
  \(
  \NH[X]+\NH[X]^\top = \BH[X]-\VH[X]
  \).
  In particular, with \( \NH:=\NH[A],\ \NHD:=\NH[\Delta A] \) and the
  analogous notation for \(\VH,\BH\) and \(\VHD,\BHD\), the corresponding
  identities are  
  \begin{align}
    &\NH+\NH^\top=\BH-\VH,
    \label{eq:h-ibp-H}\\
    &\NHD+\NHD^\top=\BHD-\VHD.
    \label{eq:h-ibp-delta}
  \end{align}
\end{lemma}
\begin{proof}
  Applying integration by parts to
  \(\partial_{x_k}(\Phi_\alpha^\top H X_k\Phi_\beta)\) and using
  \(H X_k=(H X_k)^\top\) gives the direction-wise identity. Summing over
  \(k\) gives the summed identity. Applying this identity to
  \(X_k=A_k\) and \(X_k=\Delta A_k\), using
  \Cref{lem:h-metric-sym,lem:flux-defect-h-sym}, gives
  \cref{eq:h-ibp-H,eq:h-ibp-delta}.
\end{proof}

\begin{lemma}[Energy-compatibility identity in I--EIG]
  \label{lem:i-eig-compat}
  In the I--EIG system, the energy-compatibility identity
  holds:
  \begin{equation}
    \UHAT^\top \bigl( \dMH+\VH+\BHD \bigr) \UHAT = 0 .    
    \label{eq:i-eig-compat}
  \end{equation}
\end{lemma}

\begin{proof}
  In the I--EIG system, the volume residual
  \(\partial_t \U+A_k(\U)\partial_{x_k}\U\) vanishes pointwise, so
  \Cref{prop:h-compat} applies.  Substituting
  \(\U=\sum_{\alpha=1}^{\infty}\Phi_\alpha\UHAT_\alpha\) into
  \cref{eq:h-compat} and integrating over \(\Omega\) gives
  \(
    \UHAT^\top\NHD\UHAT
    + \tfrac12\UHAT^\top
    \bigl(\dMH+\VH+\VHD\bigr)\UHAT =0
  \).
  Since
  \(
    \UHAT^\top\NHD\UHAT
    =
    \tfrac12\UHAT^\top
    \bigl(\NHD+\NHD^\top\bigr)\UHAT
  \),
  \cref{eq:h-ibp-delta} gives \cref{eq:i-eig-compat}.
\end{proof}

\subsection{Physical energy exchange structure}
\label{subsec:h-orth}
To clarify the modal exchange structure implied by the two identities,
we introduce an orthogonalization with respect to the \(H\)-metric
based on the Cholesky factorization.
\begin{definition}[$H$-orthogonalization]
  \label{def:h-orth}
  Let \(R=R(t)\) be the upper triangular Cholesky factor of \(\MH\),
  so that \( \MH=R^\top R \).  Define \(\UHATO := R\UHAT\) and
  \(X^{(o)} := R^{-\top}XR^{-1}\).  Define the Cholesky-frame
  connection \(S\) by \( \dot R = S R \).  Equivalently, \( S=\dot R
  R^{-1} \).
\end{definition}

The Cholesky frame evolves with the state-dependent mass matrix as
follows.
\begin{lemma}[Cholesky-frame connection identities]
  \label{lem:chol-connection}
  Let \(\dMHO:=R^{-\top}\dMH R^{-1}\).  Then
  \begin{equation}
    S+S^\top=\dMHO .
    \label{eq:chol-connection-sym}
  \end{equation}
  Equivalently,
  \(S=\operatorname{triu}(\dMHO)-\tfrac12\operatorname{diag}(\dMHO)\),
  where \(\operatorname{triu}\) keeps the upper triangular part including
  the diagonal, and \(\operatorname{diag}\) keeps only the diagonal part.
  The coefficient derivatives satisfy
  \begin{equation}
    \dUHATO = R\dot\UHAT + S\UHATO .
    \label{eq:uhato-velocity}
  \end{equation}
\end{lemma}
\begin{proof}
  Differentiating \(\MH=R^\top R\) and using
  \(S=\dot R R^{-1}\) gives
  \cref{eq:chol-connection-sym}.
  Since \(S\) is upper triangular, the equivalent form follows.
  Differentiating \(\UHATO=R\UHAT\) gives
  \cref{eq:uhato-velocity}.
\end{proof}

The two identities take the following form
in the \(H\)-orthogonalized variables.
\begin{lemma}[Orthogonalized \(H\)--IBP and energy-compatibility identities]
  \label{lem:h-orth-identities}
  Under the \(H\)-orthogonalization in \Cref{def:h-orth},
  the \(H\)--IBP identity becomes
  \begin{equation}
    \NHO+\NHO^\top=\BHO-\VHO .
    \label{eq:h-ibp-orth}
  \end{equation}
  The energy-compatibility identity
  \cref{eq:i-eig-compat} becomes
  \begin{equation}
    \UHATO{}^\top C_H \UHATO =0,
    \label{eq:h-compat-orth}
  \end{equation}
  where
  \begin{equation}
    C_H := \tfrac12\left(S+S^\top+\VHO+\BHOD\right) .
    \label{eq:ch-orth}
  \end{equation}
\end{lemma}
\begin{proof}
  Applying the congruence transformation
  \(X^{(o)}=R^{-\top}XR^{-1}\) to
  \cref{eq:h-ibp-H} gives \cref{eq:h-ibp-orth}.
  Substituting \(\UHAT=R^{-1}\UHATO\) and
  \cref{eq:chol-connection-sym} into
  \cref{eq:i-eig-compat} gives
  \cref{eq:h-compat-orth}, with \(C_H\).
\end{proof}

\Cref{eq:h-compat-orth} says that \(C_H\UHATO\) is orthogonal
to \(\UHATO\).  Hence, its action on \(\UHATO\) can be represented by an
antisymmetric matrix.
\begin{lemma}[Antisymmetric lift of the energy-compatibility identity]
  \label{lem:skew-lift}
  Let \(C_H\) be defined by \cref{eq:ch-orth}, and assume
  \(\UHATO\ne0\).  Define
  \begin{equation}
    K_H
    :=
    \frac{
      C_H\UHATO(\UHATO)^\top
      -
      \UHATO(\UHATO)^\top C_H
    }{
      |\UHATO|^2
    }.
    \label{eq:kh-lift}
  \end{equation}
  Then \(K_H^\top=-K_H\) and \(K_H\UHATO=C_H\UHATO\).
\end{lemma}
\begin{proof}
  By \Cref{lem:h-orth-identities},
  \(C_H=C_H^\top\) and
  \(\UHATO{}^\top C_H\UHATO=0\).
  Substitution into \cref{eq:kh-lift} gives both identities.
\end{proof}
The lift is not unique: any antisymmetric \(Z\) with \(Z\UHATO=0\) can be
added to \(K_H\).  We omit this freedom because it has
no effect on the coefficient equation.
This antisymmetric lift is the key
step for deriving the pairwise modal energy exchange and,
consequently, total energy conservation.
\begin{theorem}[Equations of orthogonalized coefficients and modal energy]
  \label{thm:i-eig-exchange}
  Let \(K_H\) be defined by \cref{eq:kh-lift}. Set
  \(
    \BHOG :=\BHO+\BHOD.
  \)
  Define \(E_\alpha := \tfrac12|\UHATO_\alpha|^2\).  Then for
  any \(\alpha \in \mathbb{N}\),
  the orthogonalized coefficient vector and the modal energy satisfy
  \begin{align}
    \dUHATO_\alpha ={}& -\sum_{\beta=1}^{\infty} \bigl(
    \NHOSKW-S_{\mathrm{skw}}-K_H \bigr)_{\alpha\beta}\UHATO_\beta
    -\tfrac12\sum_{\beta=1}^{\infty} \BHOG_{\alpha\beta}\UHATO_\beta,
    \label{eq:i-eig-state}\\
    \dot E_\alpha ={}& \sum_{\beta=1}^{\infty} P_{\alpha\beta} -
    \tfrac12\sum_{\beta=1}^{\infty} \UHATO_\alpha{}^\top
    \BHOG_{\alpha\beta}\UHATO_\beta,
    \label{eq:i-eig-mode-energy}
  \end{align}
  where \( P_{\alpha\beta} := - \UHATO_\alpha{}^\top \bigl(
  \NHOSKW-S_{\mathrm{skw}}-K_H \bigr)_{\alpha\beta} \UHATO_\beta \).
  Moreover, \(P_{\alpha\beta}=-P_{\beta\alpha} \) and \(
  \sum_{\alpha,\beta=1}^{\infty}P_{\alpha\beta}=0 \).  Consequently,
  the total energy \(E:=\sum_{\alpha=1}^{\infty}E_\alpha\) satisfies
  \begin{align}
    \dot E &= -\tfrac12\sum_{\alpha=1}^{\infty}\sum_{\beta=1}^{\infty}
    \UHATO_\alpha{}^\top \BHOG_{\alpha\beta} \UHATO_\beta .
    \label{eq:i-eig-total-energy}
  \end{align}
\end{theorem}
\begin{proof}
  From \Cref{def:h-orth} and \cref{eq:i-eig-ode},
  \(
    \dUHATO+(\NHO-S)\UHATO=0
  \).
  Using \cref{eq:h-ibp-orth,eq:ch-orth}, and
  \(\BHOG=\BHO+\BHOD\), this becomes
  \(
    \dUHATO
    = -(\NHOSKW-S_{\mathrm{skw}})\UHATO
    + C_H\UHATO
    - \tfrac12\BHOG\UHATO
  \).
  \Cref{lem:skew-lift} gives \(C_H\UHATO=K_H\UHATO\), yielding
  \cref{eq:i-eig-state}.  Differentiating \(E_\alpha\) and using
  \cref{eq:i-eig-state} gives \cref{eq:i-eig-mode-energy}.
  Since \(\NHOSKW\), \(S_{\mathrm{skw}}\), and \(K_H\) are antisymmetric,
  \(P_{\alpha\beta}=-P_{\beta\alpha}\), hence
  \(\sum_{\alpha,\beta}P_{\alpha\beta}=0\).  Summing over \(\alpha\)
  gives \cref{eq:i-eig-total-energy}.
\end{proof}

We consider neighboring elements \(\Omega\) and \(\Omega_f\) sharing
the interface \(\partial\Omega^f\).  Unbracketed quantities are viewed
from \(\Omega\), while quantities viewed from \(\Omega_f\) carry the
subscript \([\Omega_f]\).  Let \(n_f\) be the outward unit normal of
\(\Omega\) on \(\partial\Omega^f\); \(n_{f[\Omega_f]}=-n_f\).
By \Cref{lem:i-eig-no-trace}, the trace
is single-valued on the interface:
\(
  \U_{[\Omega]}|_{\partial\Omega^f}
  =
  \U_{[\Omega_f]}|_{\partial\Omega^f}
\).
Decompose the boundary operator facewise as
\(\BHOG=\sum_f\BHOG{\!\!\!}^f\).  From
\cref{eq:i-eig-mode-energy}, the \(f\)-face contributions to the
\(\alpha\)-th modal energy of \(\Omega\) and the \(\gamma\)-th modal
energy of \(\Omega_f\) are
\begin{equation}
  \begin{aligned}
    \mathcal B_\alpha^f
    :=-\tfrac12 \sum_{\beta} \UHATO_\alpha{}^\top
    (\BHOG{}^{\!\!\!f})_{\alpha\beta} \UHATO_\beta,~~
    \mathcal B_{[\Omega_f]\gamma}^f
    :=-\tfrac12 \sum_{\beta}
    \UHATO_{[\Omega_f]\gamma}{}^\top
    (\BHOG{}^{\!\!\!f}_{[\Omega_f]})_{\gamma\beta}
    \UHATO_{[\Omega_f]\beta}.
  \end{aligned}
  \label{eq:face-energy-parts}
\end{equation}
We now extend the intraelement pairwise exchange established in
\Cref{thm:i-eig-exchange} to neighboring elements.
\begin{theorem}[Interelement pairwise exchange and domain energy balance]
  \label{prop:interface-conservation}
  Under the assumptions of \Cref{thm:i-eig-exchange}, let
  \(G_k(\U)\) be the symmetric energy-flux matrix defined by
  \cref{eq:energy-flux-form}, and set
  \(G_{n_f}(\U):=G_k(\U)n_{f,k}\).
  For each pair of modes \(\alpha\) in \(\Omega\) and \(\gamma\) in
  \(\Omega_f\), define
  \begin{equation}
    P_{\alpha,[\Omega_f]\gamma}^f
    :=
    -\tfrac12
    {\textstyle \int_{\partial\Omega^f}}
    \bigl(\Phi_\alpha^{(o)}\UHATO_\alpha\bigr)^\top
    G_{n_f}(\U)
    \bigl(\Phi_{[\Omega_f]\gamma}^{(o)}
    \UHATO_{[\Omega_f]\gamma}\bigr)\,dS .
    \label{eq:interface-pair-exchange}
  \end{equation}
  Then
  \( \mathcal B_\alpha^f
  = \sum_\gamma P_{\alpha,[\Omega_f]\gamma}^f \)
  and
  \( \mathcal B_{[\Omega_f]\gamma}^f
  = \sum_\alpha P_{[\Omega_f]\gamma,\alpha}^f \),
  where
  \( P_{[\Omega_f]\gamma,\alpha}^f
  =-P_{\alpha,[\Omega_f]\gamma}^f \).
  Consequently,
  \begin{equation}
    \sum_\alpha \mathcal B_\alpha^f
    +
    \sum_\gamma \mathcal B_{[\Omega_f]\gamma}^f
    =0
    \label{eq:interface-exchange-cancellation}
  \end{equation}
  on every internal interface.  Hence, for any target domain
  \(\mathcal D\) formed by a union of elements,
  \begin{equation}
    \frac{d}{dt}
    \sum_{\Omega\subset\mathcal D}\sum_\alpha E_{\Omega,\alpha}
    =
    \mathcal B_{\partial\mathcal D},
    \label{eq:domain-energy-balance}
  \end{equation}
  where \(\mathcal B_{\partial\mathcal D}\) denotes the total contribution
  from the external boundary \(\partial\mathcal D\).
\end{theorem}
\begin{proof}
  The single-valued trace on \(\partial\Omega^f\), together with
  \cref{eq:face-energy-parts,eq:interface-pair-exchange}, gives the two
  modal decompositions of \(\mathcal B_\alpha^f\) and
  \(\mathcal B_{[\Omega_f]\gamma}^f\).
  Symmetry of \(G_{n_f}\) and \(n_{f[\Omega_f]}=-n_f\) give
  \(P_{[\Omega_f]\gamma,\alpha}^f
  =-P_{\alpha,[\Omega_f]\gamma}^f\), and hence
  \cref{eq:interface-exchange-cancellation}.
  Combining this cancellation with \Cref{thm:i-eig-exchange} and summing
  over all elements in \(\mathcal D\) gives
  \cref{eq:domain-energy-balance}.
\end{proof}

\section{Spatial discretization and its structural defect}
\label{sec:discrete-defect}
Starting from the native \(p\)-Galerkin system in
\Cref{def:native-pg}, this section constructs quadrature-based
operators. We first obtain an SPD discrete \(H\)-metric mass matrix
and enforce the \(H\)--SBP identity through a consistent projection
of the \(H\)-weighted matrix fields while retaining nominal
\(O(h^{p+1})\) consistency. Finally, we identify the remaining
energy-compatibility defect.

\subsection{Discrete metric and quadrature layers}
\label{subsec:discrete-h}
Since projecting \(H(\Up)\) need not preserve pointwise SPD, we
evaluate it directly at quadrature nodes and determine the quadrature
degree required for \(O(h^{p+1})\) mass-matrix consistency.

\begin{definition}[Quadrature-pointwise evaluation of the discrete metric]
  \label{def:hq-metric}
  For quadrature nodes \(\{x_q\}_{q=1}^{N_q}\), define \( H_q :=
  H\bigl(\Up(x_q)\bigr) \).  Since the discrete metric is used only
  through these nodal values, each \(H_q\) is SPD under
  \Cref{ass:h-bounds}.
\end{definition}

\begin{lemma}[SPD property of the quadrature-defined \(H\)-metric mass matrix]
  \label{lem:hq-mass-spd}
  Let \(Q_r\) be an \(r\)-exact volume quadrature rule with positive
  weights, and assume that \(r\ge 2p\).  Let \(H_q\) be defined by
  \Cref{def:hq-metric}.  Then the mass matrix
  defined by
  \(
    \widetilde M_{H,\alpha\beta}^{(r)}
    :=
    Q_r\!\left[\Phi_\alpha^\top H_q\Phi_\beta\right]
  \)
  is SPD.
\end{lemma}
\begin{proof}
  Since \(H_q\) is symmetric, so is \(\widetilde M_H^{(r)}\).
  For \(c\in\mathbb R^{mN_p}\), set
  \(v:=\Phi c\in(V^p)^m\). Then
  \(
    c^\top \widetilde M_H^{(r)}c
    =
    \textstyle{\sum_q} w_q\,v(x_q)^\top H_qv(x_q)
    \ge 0
  \).
  If this quantity vanishes, the SPD property of \(H_q\) and
  \(w_q>0\) imply \(v(x_q)=0\) for all \(q\).
  Since \(r\ge2p\), \(Q_r\) is exact for \(|v|^2\in V^{2p}\), so
  \(
    \|v\|_{L^2(\Omega)}^2=Q_r[|v|^2]=0
  \).
  Hence, \(v=0\), and therefore \(c=0\).
  Thus, \(c^\top\widetilde M_H^{(r)}c>0\) for every \(c\ne0\), and
  \(\widetilde M_H^{(r)}\) is SPD.
\end{proof}
The following proposition gives the stronger quadrature requirement
needed for \(O(h^{p+1})\) mass-matrix consistency.
\begin{proposition}[Required quadrature degree for metric-mass construction]
  \label{prop:h-mass-degree}
  Let \(Q_{3p}\) be a \(3p\)-exact volume quadrature rule, and let
  \(H_q\) be defined by \Cref{def:hq-metric}.
  Define the quadrature and exact-integration mass matrices by
  \begin{align}
    \MHAPP_{\alpha\beta} &:= Q_{3p}\!\left[\Phi_\alpha^\top H_q
      \Phi_\beta\right],
    \label{eq:mh-q3p}\\
    \MHref_{\alpha\beta} &:= \textstyle{\int_\Omega} \Phi_\alpha^\top
    H(\Up)\Phi_\beta\,dV.
    \label{eq:mh-ref}
  \end{align}  
  Then for every $\alpha,\beta=1,\dots,N_p$, \(
  \|\MHAPP_{\alpha\beta}-\MHref_{\alpha\beta}\|_F \le C
  h^{p+1} \), with a constant $C$ independent of $h$, $\alpha$, and
  $\beta$.
\end{proposition}
\begin{proof}
  Define \( \Delta\MH:=\MHAPP-\MHref \),
  \( H_p:=\Pi_{L^2,p}(H(\Up)) \), and
  \( \delta H:=H(\Up)-H_p \).
  Since \(\phi_\alpha\phi_\beta H_p\) has degree at most \(3p\),
  \( (\Delta\MH)_{\alpha\beta} = \Bigl(Q_{3p}-\int_\Omega\Bigr)
    [\phi_\alpha\phi_\beta\delta H] \)
  for each block.
  Applying Cauchy--Schwarz gives
  \begin{align*}
    \|(\Delta\MH)_{\alpha\beta}\|_F
    &\le
    \|\delta H\|_{L^\infty(\Omega);F}
    \left(
      Q_{3p}[\phi_\alpha^2]^{1/2}
      Q_{3p}[\phi_\beta^2]^{1/2}
      +
      \|\phi_\alpha\|_{L^2(\Omega)}
      \|\phi_\beta\|_{L^2(\Omega)}
    \right).
  \end{align*}
  The basis \(\{\phi_\alpha\}\subset V^p\) is \(L^2\)-orthonormal
  by \Cref{def:vp-expansion}.
  Hence, \(\|\phi_\alpha\|_{L^2(\Omega)}^2=1\), and
  \(Q_{3p}[\phi_\alpha^2]=1\) by the \(3p\)-exactness of \(Q_{3p}\).
  Thus, \(
    \|(\Delta\MH)_{\alpha\beta}\|_F \le 2\|\delta H\|_{L^\infty(\Omega);F} \).
  \Cref{ass:field-proj-error} gives
  \( \|H(\Up)-\Pi_{L^2,p}(H(\Up))\|_{L^\infty(\Omega);F} \le Ch^{p+1} \).
  Therefore,
  \(
    \|\MHAPP_{\alpha\beta}-\MHref_{\alpha\beta}\|_F
    \le
    Ch^{p+1}
  \).
\end{proof}

\subsection{Guarantee of the $H$--SBP identity}
\label{subsec:h-sbp-discrete}
This subsection constructs quadrature-based operators satisfying the
\(H\)--SBP identity.
We first state the structural conditions required for this identity.
\begin{proposition}[Structural condition for \(H\)--SBP]
  \label{prop:h-sbp-layer}
  Let \(Y_k\) be a matrix-valued field for a fixed
  direction \(k\).  Let \(Q_r\) be a volume quadrature rule and
  \(Q_q^{\partial\Omega}\) be a face quadrature rule.  Define
  \( N\ksup_{\alpha\beta} := Q_r[\Phi_\alpha^\top Y_k\partial_{x_k}\Phi_\beta] \),
  \( V\ksup_{\alpha\beta} := Q_r[\Phi_\alpha^\top(\partial_{x_k}Y_k)\Phi_\beta]\),
  and
  \( B\ksup_{\alpha\beta} := Q_q^{\partial\Omega} [\Phi_\alpha^\top Y_k\Phi_\beta n_k]\).
  Suppose that the following two conditions hold:
  \begin{enumerate}
  \item[\rm(i)] for all \(x,y\in(V^p)^m\), with \(\psi_k:=y^\top
    Y_kx\), \( Q_r[\partial_{x_k}\psi_k] = Q_q^{\partial\Omega}[\psi_k
      n_k].  \)
  \item[\rm(ii)] \(Y_k(x_q)^\top=Y_k(x_q)\) at every volume quadrature
    node.
  \end{enumerate}
  Then the \(H\)--SBP identity holds:
  \( N\ksup+N\ksup{}^\top = B\ksup-V\ksup \).
  Consequently, if these conditions hold for every \(k\), then \(
  N+N^\top=B-V, \) where \(N:=\sum_kN\ksup\), and similarly for \(V\)
  and \(B\).
\end{proposition}
\begin{proof}
  For \(x,y\in(V^p)^m\), 
  \( y^\top\!\left(
  N\ksup+N\ksup{}^\top-B\ksup+V\ksup
  \right)x
  =
  Q_r[\partial_{x_k}(y^\top Y_kx)]
  -
  Q_q^{\partial\Omega}[(y^\top Y_kx)n_k]
  +
  Q_r[(\partial_{x_k}y)^\top(Y_k^\top-Y_k)x]
  \).
  Condition \rm(i) cancels the first two terms, and condition \rm(ii)
  makes the last term vanish.    Hence, \(N\ksup+N\ksup{}^\top=B\ksup-V\ksup\).
  Summing over \(k\) gives the stated summed identity.
\end{proof}
We enforce condition~\textup{(i)} by exactness rather than by
cancellation of quadrature errors.
For the generally nonpolynomial fields \(H(\Up)X_k(\Up)\), we use one
polynomial projection in both the volume and face operators.
The projection and operator construction use the two quadrature layers
defined below.

\begin{definition}[Construction quadrature layer]
  \label{def:q3p-layer}
  Let \(Q_{3p}\) be a \(3p\)-exact volume quadrature rule and let
  \(Q_{3p+1}^{f}\) be a \((3p+1)\)-exact rule on each face
  \(f\subset\partial\Omega\).
  The collection is denoted by
  \(Q_{3p+1}^{\partial\Omega}\), with the same face rule used by
  neighboring elements.
  These rules form the construction layer for the volume, boundary,
  and interelement terms in the SBP--SAT-type coupling
  \cite{DelReyFernandezHickenZingg2018SAT}.
\end{definition}
\begin{definition}[High-order quadrature rule]
  \label{def:qhi}
  For \(n\ge3p+1\), let
  \(Q_{\mathrm{hi}(n)}:=Q_n\) be an \(n\)-exact volume quadrature rule.
  It is used to evaluate nonlinear coefficient fields in the
  projection step.
\end{definition}
\begin{remark}[One-degree offset between volume and face quadratures]
  \label{rem:quadrature-degree-offset}
  For \(x,y\in(V^p)^m\) and
  \(Y_{X,k,p+1}\in V^{p+1}\), the boundary integrand
  \(  \psi_k:=y^\top Y_{X,k,p+1}x \)
  belongs to \(V^{3p+1}\), whereas
  \(\partial_{x_k}\psi_k\in V^{3p}\).
  Hence, the face quadrature requires one degree higher exactness than
  the volume quadrature.
\end{remark}
Using these two quadrature layers, we define the practical operator
construction.

\begin{definition}[Coupled Projection Galerkin (CPG)]
  \label{def:cpg}
  For \(X\in\{A,\Delta A\}\) and each \(k\), set
  \(
    Y_{X,k}(\Up):=H(\Up)X_k(\Up)
  \)
  and define
  \[
    Y_{X,k,p+1}
    :=
    \Pi_{L^2,p+1}^{\,Q_{\mathrm{hi}(n)}}
    [Y_{X,k}(\Up)]
    \qquad\in V^{p+1}(\Omega)^{m\times m}.
  \]
  The same projected field is used in the volume and face operators:
  \begin{align*}
    \NHXAPPk_{\alpha\beta}
    &:=
    Q_{3p}\!\left[
      \Phi_\alpha^\top
      Y_{X,k,p+1}\partial_{x_k}\Phi_\beta
    \right],
    \qquad
    \VHXAPPk_{\alpha\beta}
    :=
    Q_{3p}\!\left[
      \Phi_\alpha^\top
      (\partial_{x_k}Y_{X,k,p+1})\Phi_\beta
    \right],
    \\
    \BHXAPPk_{\alpha\beta}
    &:=
    Q_{3p+1}^{\partial\Omega}\!\left[
      \Phi_\alpha^\top
      Y_{X,k,p+1}\Phi_\beta n_k
    \right].
  \end{align*}
  Set
  \( \NHXAPP:=\sum_k\NHXAPPk \),
  \( \VHXAPP:=\sum_k\VHXAPPk \),
  and
  \( \BHXAPP:=\sum_k\BHXAPPk \).
  For \(X=A\), write
  \( \NHAPP:=\NHAAPP \),
  \( \VHAPP:=\VHAAPP \),
  and
  \( \BHAPP:=\BHAAPP \).
  For \(X=\Delta A\), write
  \(\NHDAPP:=\NHDAAPP \),
  \(\VHDAPP:=\VHDAAPP \),
  and
  \(\BHDAPP:=\BHDAAPP\).
  Together with \(\MHAPP\) in \cref{eq:mh-q3p}, this defines
  CPG\((n,p)\).
\end{definition}

\begin{corollary}[\(H\)--SBP property of CPG]
  \label{cor:cpg-h-sbp}
  For \(X\in\{A,\Delta A\}\), the CPG construction in \Cref{def:cpg} satisfies
  \begin{equation}
    \NHXAPP+\NHXAPP{}^\top=\BHXAPP-\VHXAPP.
    \label{eq:cpg-h-sbp}
  \end{equation}
\end{corollary}
\begin{proof}
  By \Cref{rem:quadrature-degree-offset}, the construction quadratures
  satisfy condition~\textup{(i)} of \Cref{prop:h-sbp-layer}.
  By \Cref{lem:h-metric-sym,lem:flux-defect-h-sym},
  \(Y_{X,k}(\Up)\) is symmetric, and the componentwise projection
  preserves this symmetry, so condition~\textup{(ii)} also holds.
  Hence, \Cref{prop:h-sbp-layer} gives \cref{eq:cpg-h-sbp}.
\end{proof}

\begin{proposition}[Nominal accuracy of CPG operators]
  \label{prop:cpg-accuracy}
  For \(X\in\{A,\Delta A\}\), let
  \(\NHXAPPk\), \(\VHXAPPk\), and \(\BHXAPPk\) be the CPG operators
  defined in \Cref{def:cpg}, and let
  \(\NHXk\), \(\VHXk\), and \(\BHXk\) be the corresponding
  exact-integration operators constructed with
  \(Y_{X,k}(\Up):=H(\Up)X_k(\Up)\).
  Under
  \Cref{ass:shape-regular,ass:up-bounds,ass:field-proj-error},
  for each fixed pair of modes \(\alpha,\beta\),
  \begin{equation}
    \|\NHXAPPk_{\alpha\beta}-\NHXk_{\alpha\beta}\|_F
    +
    \|\VHXAPPk_{\alpha\beta}-\VHXk_{\alpha\beta}\|_F
    +
    \|\BHXAPPk_{\alpha\beta}-\BHXk_{\alpha\beta}\|_F
    \le Ch^{p+1},
    \label{eq:cpg-accuracy}
  \end{equation}
  where \(C\) is independent of \(h\).
\end{proposition}
\begin{proof}
  Set
  \(
    \tau_k:=Y_{X,k}(\Up)-Y_{X,k,p+1}
  \).
  By \Cref{ass:field-proj-error},
  \(
    \|\tau_k\|_{L^\infty(\Omega);F}
    +
    h\|\nabla\tau_k\|_{L^\infty(\Omega);F}
    \le Ch^{p+2}
  \).
  By \Cref{def:cpg,rem:quadrature-degree-offset}, the projected
  contributions are integrated exactly.  Hence,
  \(\NHXAPPk_{\alpha\beta}-\NHXk_{\alpha\beta}
  = -\int_\Omega \Phi_\alpha^\top\tau_k\partial_{x_k}\Phi_\beta\,dV\),
  \( \VHXAPPk_{\alpha\beta}-\VHXk_{\alpha\beta}
  = -\int_\Omega \Phi_\alpha^\top(\partial_{x_k}\tau_k)\Phi_\beta\,dV\),
  \( \BHXAPPk_{\alpha\beta}-\BHXk_{\alpha\beta}
  = -\int_{\partial\Omega}  \Phi_\alpha^\top\tau_k\Phi_\beta n_k\,dS\).
  The first two estimates follow from the bounds on
  \(\tau_k\) and \(\nabla\tau_k\), together with the standard basis
  scaling.  The third follows from the face trace estimate on
  shape-regular elements.  Thus, all three terms are
  \(O(h^{p+1})\), which gives \cref{eq:cpg-accuracy}.
\end{proof}

\subsection{Breakdown of the energy-compatibility identity}
\label{subsec:compat-defect}
Although the \(H\)--SBP identity recovers the discrete integration-by-parts
structure, it does not enforce the energy-compatibility identity
\cref{eq:i-eig-compat}, which also contains the mass-matrix evolution
(MME). We distinguish the MME along the actual finite-mode trajectory
from that induced by the PDE volume direction as follows.
\begin{definition}[True and PDE-induced MMEs]
  \label{def:mass-mme}
  Let \(\Up(t)\in(V^p)^m\), and let \(\MHref(\Up)\) and
  \(\MHAPP(\Up)\) denote the exact-integration and quadrature-based
  \(H\)-metric mass matrices, respectively.
  Their true MMEs are
  \[
    \dMHref
    :=
    D_{\Up}\MHref(\Up)[\partial_t\Up],
    \qquad
    \dMHAPP
    :=
    D_{\Up}\MHAPP(\Up)[\partial_t\Up].
  \]
  Replacing \(\partial_t\Up\) in the chain rule by the volume part of
  \cref{eq:ql-system} defines the corresponding PDE-induced MMEs:
  \begin{subequations}
    \label{eq:mme-definitions}
    \begin{align}
      \dMHV
      &:=
      D_{\Up}\MHref(\Up)
      \bigl[-A_k(\Up)\partial_{x_k}\Up\bigr],
      \label{eq:mme-pde}\\
      \dMHVAPP
      &:=
      D_{\Up}\MHAPP(\Up)
      \bigl[-A_k(\Up)\partial_{x_k}\Up\bigr].
      \label{eq:mme-quad}
    \end{align}
  \end{subequations}  
\end{definition}

The true and PDE-induced MMEs are the directional derivatives of the
mass matrix along \(\partial_t\Up\) and
\(-A_k(\Up)\partial_{x_k}\Up\), respectively, at the same state.
They coincide in the I--EIG system but not in general after finite-mode
truncation.
The PDE-induced MME nevertheless satisfies the following
exact-integration identity.
\begin{lemma}[PDE-induced energy-compatibility-type identity]
  \label{lem:pde-mme-identity}
  Let \(\Up\in(V^p)^m\) be a resolved finite-mode state.  Let
  \(\dMHV\) be the PDE-induced MME defined by \cref{eq:mme-pde},
  and let \(\VH\) and
  \(\BHD\) be the corresponding exact-integration native
  \(p\)-Galerkin quantities evaluated from the same state \(\Up\).
  Then
  \begin{equation}
    \UHATp^\top
    \bigl(
      \dMHV+\VH+\BHD
    \bigr)
    \UHATp
    =
    0 .
    \label{eq:pde-mme-compat}
  \end{equation}
\end{lemma}
\begin{proof}
  Set  \(
    \delta_{\rm vol}\Up
    :=
    -A_k(\Up)\partial_{x_k}\Up
  \).
  Applying the algebraic calculation in the proof of
  \Cref{prop:h-compat} to \(\Up\), with
  \(\partial_t\Up\) replaced by \(\delta_{\rm vol}\Up\), gives
  \[
    \begin{aligned}
      &\Up^\top H(\Up)\Delta A_k(\Up)\partial_{x_k}\Up \\
      &\quad
      +
      \tfrac12 \Up^\top
      \left[
        D_{\Up} H(\Up)[\delta_{\rm vol}\Up]
        +
        \partial_{x_k}
        \{H(\Up)(A_k(\Up)+\Delta A_k(\Up))\}
      \right]
      \Up
      =
      0 .
    \end{aligned}
  \]
  Expanding \(\Up\), integrating exactly, and applying
  \cref{eq:h-ibp-delta} gives \cref{eq:pde-mme-compat}.
\end{proof}
Thus, in the finite-mode exact-integration system,
\cref{eq:pde-mme-compat} holds for the PDE-induced MME, whereas the
true MME need not satisfy the energy-compatibility identity. The
corresponding quadrature identity is also not guaranteed.

\begin{proposition}[Breakdown of energy compatibility in the discrete system]
  \label{prop:cpg-compat-defect}
  Assume that the operator construction in the discrete system
  satisfies the \(H\)--SBP identities for \(H(\Up)A_k(\Up)\) and
  \(H(\Up)\Delta A_k(\Up)\).  Then in general,
  \begin{equation}
    \UHATp^\top
    \bigl(
      \dMHAPP+\VHAPP+\BHDAPP
    \bigr)
    \UHATp
    \not\equiv 0 .
    \label{eq:break-compatiblity}
  \end{equation}
\end{proposition}
\begin{proof}
  Let \(\VH\) and \(\BHD\) denote the corresponding exact-integration
  operators at the same state \(\Up\).
  By \Cref{lem:pde-mme-identity},
  \( \UHATp^\top(\dMHV+\VH+\BHD)\UHATp=0 \). Hence,
  \[
    \UHATp^\top(\dMHAPP+\VHAPP+\BHDAPP)\UHATp
    =
    \UHATp^\top
    \Bigl[
      (\dMHAPP-\dMHV)
      +(\VHAPP-\VH)
      +(\BHDAPP-\BHD)
    \Bigr]\UHATp .
  \]
  Adding and subtracting \(\dMHref\) gives
  \[
  \begin{aligned}
    &\UHATp^\top \bigl( \dMHAPP+\VHAPP+\BHDAPP \bigr) \UHATp
    =
    \UHATp^\top \bigl( \dMHref-\dMHV \bigr) \UHATp \\
    &\qquad
    +
    \UHATp^\top \Bigl[ (\dMHAPP-\dMHref)
    +(\VHAPP-\VH) +(\BHDAPP-\BHD) \Bigr] \UHATp .
  \end{aligned}
  \]
  The first term is the native finite-mode defect.
  By the definitions of \(\dMHref\) and \(\dMHV\) and
  \cref{eq:residual-p},
  \(
    \UHATp^\top(\dMHref-\dMHV)\UHATp
    =
    \int_\Omega
    \Up^\top D_{\Up}H(\Up)[\mathcal R_p]\Up\,dV
  \).
  Since \(\mathcal R_p\) need not vanish pointwise in the native
  \(p\)-Galerkin system, this term is not structurally forced to vanish.
  The second term is the quadrature defect, which is not forced to
  vanish by the \(H\)--SBP identities.
  Therefore, \cref{eq:break-compatiblity} holds.
\end{proof}

\section{Recovery of the energy structure}
\label{sec:h-ecf}
This section recovers the discrete energy-compatibility identity
without introducing an independent correction to the
trajectory-dependent \(H\)-orthogonal frame.
We first derive the energy-compatibility residual in the quadrature-based frame and then remove only the component of the compatibility action parallel to the coefficient vector so that the residual vanishes.
The remaining tangential action is represented by an antisymmetric lift.
This realizes the pairwise modal energy-exchange structure
and the corresponding energy balance of the I--EIG system in \Cref{sec:i-eig}.
Finally, we derive the equivalent coefficient equation in the fixed
Galerkin basis.

\subsection{Energy compatibility residual}
We apply the \(H\)-orthogonalization in \Cref{def:h-orth} to the
quadrature-based mass matrix \(\MHAPP\).
Let \(\widetilde R\) denote the corresponding upper triangular
Cholesky factor, so that
\(\MHAPP=\widetilde R^\top\widetilde R\), and set
\( \UHATOp:=\widetilde R\UHATp\) and
\( \widetilde X^{(o)} :=\widetilde R^{-\top}\widetilde X\widetilde R^{-1}\)
for any quadrature-based matrix \(\widetilde X\).
With
\( \dMHOAPP :=\widetilde R^{-\top}\dMHAPP\widetilde R^{-1} \),
the connection identities in \Cref{lem:chol-connection} give
\( \widetilde S := \dot{\widetilde R}\widetilde R^{-1}
= \operatorname{triu}(\dMHOAPP)
-\tfrac12\operatorname{diag}(\dMHOAPP) \).
\( \widetilde S+\widetilde S^\top=\dMHOAPP \),
and 
\begin{equation}
  \dUHATOp
  =
  \widetilde R\dUHATp+\widetilde S\UHATOp.
  \label{eq:true-frame-velocity}
\end{equation}
The above MME is written by using \(Q_{3p}\) rule as
\begin{equation}
  \dMHAPP
  =
  Q_{3p}\!\left[
    \Phi^\top
    D_{\Up}H(\Up)[\partial_t\Up]
    \Phi
  \right].
  \label{eq:true-mme-q3p}
\end{equation}
We define the compatibility operator and its scalar residual by
\begin{align}
  \widetilde C_H
  &:=\tfrac12\left(
    \dMHOAPP+\VHOAPP+\BHODAPP
  \right),
  \label{eq:ch-true}\\
  \rho_H
  &:=\UHATOp{}^\top\widetilde C_H\UHATOp.
  \label{eq:rhoh-true}
\end{align}
The residual \(\rho_H\) is not structurally forced to vanish by
\Cref{prop:cpg-compat-defect}.

\subsection{Recovery of the energy-compatibility identity}
\label{subsec:h-ecf-compat}

\begin{figure}[t]
  \centering
  \includegraphics[width=0.7\linewidth]{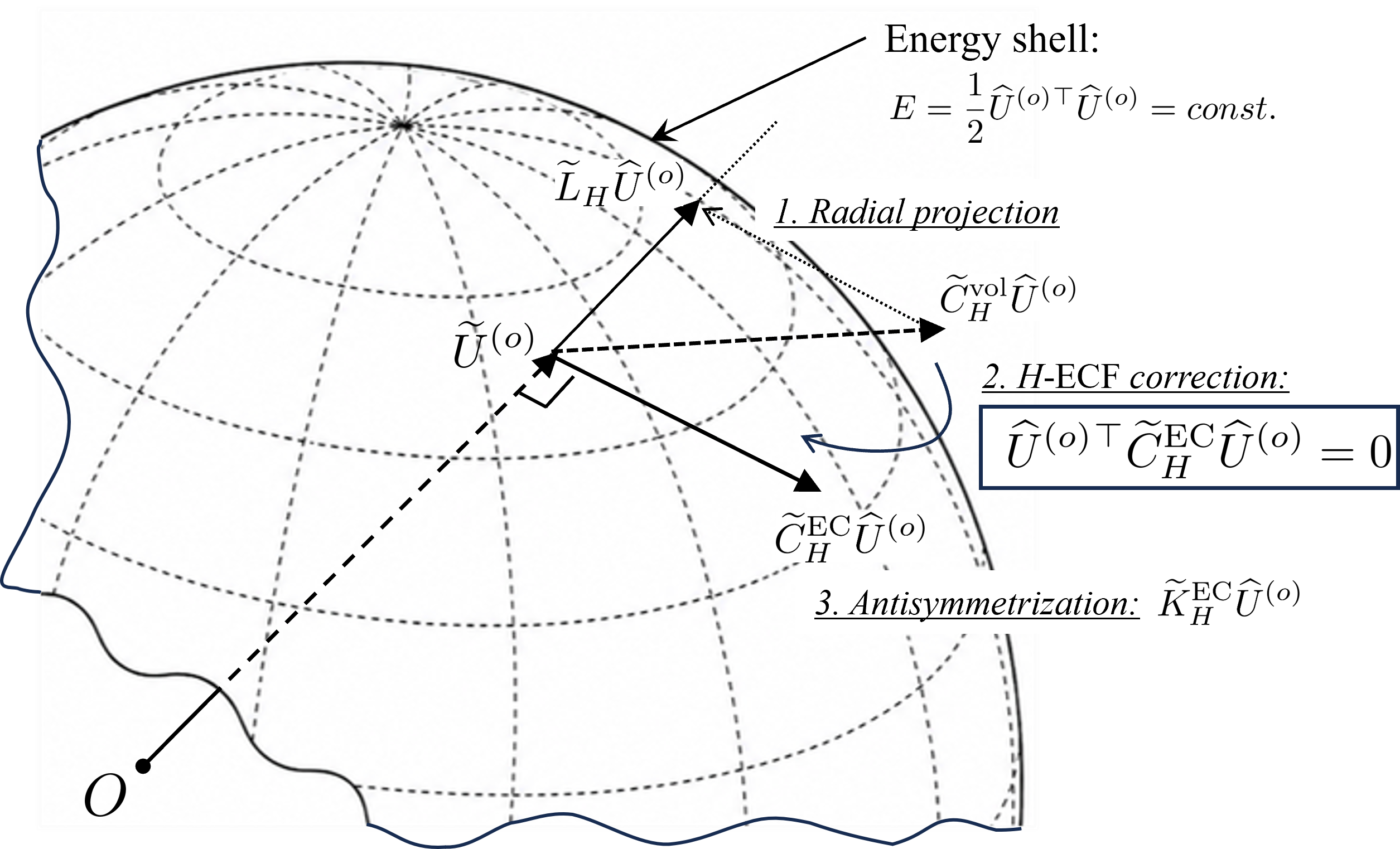}

  \caption{
    Geometric interpretation of \(H\)--ECF in the
    \(H\)-orthogonalized coefficient space.
    The correction removes the radial component
    \(\widetilde L_H\UHATOp\) of
    \(\widetilde C_H\UHATOp\), leaving
    \(\widetilde C_H^{\rm EC}\UHATOp\) tangent to the energy shell
    \(E=\tfrac12|\UHATOp|^2\).
    The tangential action is represented by the antisymmetric lift
    \(\widetilde K_H^{\rm EC}\UHATOp\).
  }
  \label{fig:h-ecf-geometry}
\end{figure}

\begin{definition}[\(H\)--Energy Compatibility Forcing (\(H\)--ECF)]
  \label{def:h-ecf}
  The \(H\)--ECF closure replaces only the compatibility action
  \(\widetilde C_H\UHATOp\) by its component orthogonal to
  \(\UHATOp\).
  It therefore removes the component that contributes to the scalar
  energy residual and retains the tangential action on the
  physical-energy shell; see \Cref{fig:h-ecf-geometry}.
\end{definition}
\begin{remark}
  The \(H\)--ECF closure introduces no independent correction to the
  quadrature-based metric evolution or its Cholesky frame.
  Along the trajectory generated by the closed coefficient equation,
  \(\MHAPP\), \(\dMHAPP\), \(\widetilde R\), and \(\widetilde S\)
  are determined by the current coefficient vector and its actual
  time derivative through the relations introduced above.
\end{remark}

The removed component is represented by the following symmetric
operator.
\begin{lemma}[Projection of the compatibility action onto the state direction]
  \label{lem:state-projection}
  Let \(\UHATOp\neq0\), and let \(\rho_H\) be defined by
  \cref{eq:rhoh-true}.  Define
  \begin{equation}
    \widetilde L_H
    :=
    \frac{\rho_H}{|\UHATOp|^4}
    \UHATOp\UHATOp{}^\top.
    \label{eq:lh-projection}
  \end{equation}
  Then \(\widetilde L_H\) is symmetric and satisfies
  \(
    \UHATOp{}^\top\widetilde L_H\UHATOp=\rho_H
  \).
\end{lemma}
\begin{proof}
  Symmetry is immediate, and substitution gives the identity.
\end{proof}

\begin{proposition}[Recovery of the scalar energy-compatibility identity]
  \label{prop:h-ecf-compat}
  Define the corrected compatibility operator by
  \( \widetilde C_H^{\rm EC} :=\widetilde C_H-\widetilde L_H \).
  Then
  \begin{equation}
    \UHATOp{}^\top\widetilde C_H^{\rm EC}\UHATOp=0.
    \label{eq:h-ecf-compatibility}
  \end{equation}
\end{proposition}
\begin{proof}
  \Cref{eq:rhoh-true,eq:lh-projection} gives \cref{eq:h-ecf-compatibility}
\end{proof}
The corrected compatibility action admits the same antisymmetric lift
as in \Cref{lem:skew-lift}.
\begin{lemma}[Antisymmetric lift after \(H\)--ECF]
  \label{lem:h-ecf-skew-lift}
  Let \(\UHATOp\neq0\), and define
  \begin{equation}
    \widetilde K_H^{\rm EC}
    :=
    \frac{
      \widetilde C_H^{\rm EC}\UHATOp\UHATOp{}^\top
      -
      \UHATOp\UHATOp{}^\top\widetilde C_H^{\rm EC}
    }{
      |\UHATOp|^2
    }.
    \label{eq:kh-ecf}
  \end{equation}
  Then
  \((\widetilde K_H^{\rm EC})^\top=-\widetilde K_H^{\rm EC}\) and
  \(
    \widetilde K_H^{\rm EC}\UHATOp
    =\widetilde C_H^{\rm EC}\UHATOp
  \).
\end{lemma}
\begin{proof}
  Apply \Cref{lem:skew-lift} using \cref{eq:h-ecf-compatibility}.
\end{proof}

\begin{theorem}[Equations of orthogonalized coefficients and modal energy under \(H\)--ECF/CPG]
  \label{thm:h-ecf-exchange}
  Let \(\widetilde K_H^{\rm EC}\) be defined by
  \cref{eq:kh-ecf}, and set
  \( \BHOGAPP:=\BHOAPP+\BHODAPP \) and 
  \( \QHGNUMO:=\QHNUMO+\BHODAPP\UHATOp\).
  Define
  \(
    E_\alpha:=\tfrac12|\UHATOp_\alpha|^2
  \).
  Then for \(\alpha=1,\ldots,N_p\),
  \begin{align}
    \dUHATOp_\alpha
    ={}&
    -\sum_{\beta=1}^{N_p}
    \left(
      \NHOAPPSKW
      -\widetilde S_{\rm skw}
      -\widetilde K_H^{\rm EC}
    \right)_{\alpha\beta}
    \UHATOp_\beta
    +\tfrac12\sum_{\beta=1}^{N_p}
    \BHOGAPP_{\alpha\beta}\UHATOp_\beta
    -\QHGNUMO_\alpha,
    \label{eq:h-ecf-state}\\
    \dot E_\alpha
    ={}&
    \sum_{\beta=1}^{N_p}P_{\alpha\beta}^{\rm EC}
    +\tfrac12\sum_{\beta=1}^{N_p}
    \UHATOp_\alpha{}^\top
    \BHOGAPP_{\alpha\beta}\UHATOp_\beta
    -\UHATOp_\alpha{}^\top\QHGNUMO_\alpha,
    \label{eq:h-ecf-mode-energy}
  \end{align}
  where
  \[
    P_{\alpha\beta}^{\rm EC}
    :=
    -\UHATOp_\alpha{}^\top
    \left(
      \NHOAPPSKW
      -\widetilde S_{\rm skw}
      -\widetilde K_H^{\rm EC}
    \right)_{\alpha\beta}
    \UHATOp_\beta.
  \]
  Moreover,
  \(P_{\alpha\beta}^{\rm EC}=-P_{\beta\alpha}^{\rm EC}\) and
  \(
    \sum_{\alpha,\beta=1}^{N_p}P_{\alpha\beta}^{\rm EC}=0
  \).
  Consequently, the total energy
  \(E:=\sum_{\alpha=1}^{N_p}E_\alpha\) satisfies
  \begin{equation}
    \dot E
    =
    \tfrac12
    \sum_{\alpha,\beta=1}^{N_p}
    \UHATOp_\alpha{}^\top
    \BHOGAPP_{\alpha\beta}\UHATOp_\beta
    -
    \sum_{\alpha=1}^{N_p}
    \UHATOp_\alpha{}^\top\QHGNUMO_\alpha.
    \label{eq:h-ecf-total-energy}
  \end{equation}
\end{theorem}
\begin{proof}
  The orthogonalized CPG equation and the \(H\)--SBP identity give
  \(
    \dUHATOp
    =
    -\left(
      \NHOAPPSKW-\widetilde S_{\rm skw}
    \right)\UHATOp
    +\widetilde C_H\UHATOp
    +\tfrac12\BHOGAPP\UHATOp
    -\QHGNUMO
  \).
  The \(H\)--ECF closure replaces
  \(\widetilde C_H\UHATOp\) by
  \(\widetilde C_H^{\rm EC}\UHATOp\).
  By \Cref{lem:h-ecf-skew-lift},
  \(
    \widetilde C_H^{\rm EC}\UHATOp
    =
    \widetilde K_H^{\rm EC}\UHATOp
  \).
  Substitution gives \cref{eq:h-ecf-state}.
  Taking the inner product of its \(\alpha\)-th block with
  \(\UHATOp_\alpha\) gives
  \cref{eq:h-ecf-mode-energy}.
  Since
  \(
    \NHOAPPSKW
    -\widetilde S_{\rm skw}
    -\widetilde K_H^{\rm EC}
  \)
  is antisymmetric,
  \(P_{\alpha\beta}^{\rm EC}=-P_{\beta\alpha}^{\rm EC}\).
  Summing \cref{eq:h-ecf-mode-energy} over \(\alpha\) gives
  \cref{eq:h-ecf-total-energy}.
\end{proof}

\Cref{thm:h-ecf-exchange} closes the volume contribution as an
antisymmetric modal exchange. At an interface, the boundary dynamics
must satisfy two requirements: neighboring elements must remain
coupled, and their energy exchange must retain the same local
energy-compatibility structure as the I--EIG interface exchange in
\Cref{prop:interface-conservation}.
Recall that
\(
  \partial\Omega_f
  :=\partial\Omega^-\cap\partial\Omega^+
\)
denotes an internal interface, and let $n$ be the common unit normal
directed from $\Omega^-$ to $\Omega^+$. We denote the two traces by
$\Up^\pm$.
Following \Cref{def:cpg}, we set
\(
  G_{k,p+1}^{\pm}
  :=Y_{A,k,p+1}^{\pm}+Y_{\Delta A,k,p+1}^{\pm}
\)
and
\(
  G_{n,p+1}^{\pm}:=G_{k,p+1}^{\pm}n_k .
\)
Note that both $G_{n,p+1}^{-}$ and $G_{n,p+1}^{+}$ are defined with
respect to the same orientation $n$, rather than the respective
outward normals of the two elements.

The boundary term in \cref{eq:h-ecf-state} is written in the
$H$-orthogonal coefficient space. To represent the same contribution
in physical space, we introduce $q^{f\pm}\in\mathbb R^m$ as the
boundary actions at each shared quadrature point.
Exact interface energy compatibility then requires
\begin{equation}
  (\Up^-)^\top q^{f-}
  +(\Up^+)^\top q^{f+}
  =0.
  \label{eq:cpg-interface-energy-compatibility}
\end{equation}

In the I--EIG system, the natural element-local action is
\(
  q_{\rm self}^{f\pm}
  =\pm\tfrac12 G_n\U^\pm
\).
Since both the interface trace and $G_n$ are single-valued, the same
action can be written equivalently as
\(
  q_{\rm cross}^{f\pm}
  =\pm\tfrac12 G_n\U^\mp
\).
The two forms generate the same I--EIG dynamics,
while the cross form makes the interelement exchange in
\Cref{prop:interface-conservation} explicit.

This equivalence is lost in CPG. The direct self form
\(
  q_{\rm self}^{f\pm}
  =\pm\tfrac12G_{n,p+1}^{\pm}\Up^\pm
\)
is element-local, so that it decouples the two neighboring elements.
In addition, it does not satisfy
\cref{eq:cpg-interface-energy-compatibility} in general.
To preserve the interelement coupling,
we inherit the cross representation of the I--EIG dynamics,
defining the CPG reference action as
\begin{equation}
  q_0^{f\pm}
  :=\pm\tfrac12G_{n,p+1}^{\mp}\Up^\mp .
  \label{eq:cpg-reference-face-action}
\end{equation}
Although this cross form restores the interelement coupling,
it does not yet satisfy exact energy compatibility.
Since $G_{n,p+1}^-$ and $G_{n,p+1}^+$ are generally different,
\cref{eq:cpg-reference-face-action}
leaves the quadrature-pointwise defect
\begin{equation}
  r_f
  :=
  (\Up^-)^\top q_0^{f-}
  +(\Up^+)^\top q_0^{f+}
  =
  (\Up^-)^\top\Delta G_{n,p+1}\Up^+ ,
  \label{eq:cpg-modal-energy-defect}
\end{equation}
where
\(
  \Delta G_{n,p+1}
  :=\tfrac12(G_{n,p+1}^--G_{n,p+1}^+)
\),
and the symmetry of $G_{n,p+1}^\pm$ has been used.
Thus, after inheriting the I--EIG cross-coupling,
the only remaining source of incompatibility is the mismatch between
the two CPG representations of $G_n$.

To remove this remaining defect,
we also apply the $H$--ECF principle to the interface.
We set
\(
  q^{f\pm}=q_0^{f\pm}+\Delta q^{f\pm}
\)
and enforce the interface energy compatibility relation
\cref{eq:cpg-interface-energy-compatibility}, which, together with
\cref{eq:cpg-modal-energy-defect}, gives
\(
  (\Up^-)^\top\Delta q^{f-}
  +(\Up^+)^\top\Delta q^{f+}
  =-r_f
\).
Among all such corrections, the unique minimum Euclidean-norm
correction is
\begin{equation}
  \Delta q^{f\pm}=-\lambda \Up^\pm,
  \qquad
  \lambda
  :=\frac{r_f}{|\Up^-|^2+|\Up^+|^2}.
\label{eq:cpg-interface-lambda}
\end{equation}

The corrected boundary action can now be related to the one-sided
numerical fluxes in \cref{eq:h-ecf-state}.
At each shared quadrature point, it is represented as
\begin{equation}
  q^{f\pm}
  =
  \pm\widehat{\mathcal G}_{H,n}^{\pm}
  \mp\tfrac12G_{n,p+1}^{\pm}\Up^\pm ,
  \label{eq:gn-face-representation}
\end{equation}
where $\widehat{\mathcal G}_{H,n}^{\pm}$ denote the one-sided
numerical fluxes.
Note that \(n\) is the outward normal of \(\Omega^-\); hence
\(q^{f-}\) corresponds directly to the facewise boundary contribution
in \cref{eq:h-ecf-state}.
Using
\cref{eq:cpg-reference-face-action,eq:cpg-interface-lambda,eq:gn-face-representation},
we obtain
\[
  \widehat{\mathcal G}_{H,n}^{\pm}
  =
  \tfrac12
  \left(
    G_{n,p+1}^-\Up^-
    +G_{n,p+1}^+\Up^+
  \right)
  \mp\lambda\Up^\pm .
\]
Thus the physical-space interface dynamics is determined before any
modal decomposition: the I--EIG structure fixes the coupled reference
action, and the $H$--ECF supplies the minimum correction required for
exact energy compatibility.

Projecting the corrected boundary actions onto the local
$H$-orthogonal bases, we define the projected interface actions by
\(
  \mathcal I^{f\pm}
  :=
  Q_{3p+1}^{f}\!\left[
    (\Phi^{(o)\pm})^\top q^{f\pm}
  \right]
\).
Since \cref{eq:cpg-interface-energy-compatibility} holds at every
shared quadrature point, these interface actions satisfy
\(
  (\UHATOp{}^-)^\top\mathcal I^{f-}
  +
  (\UHATOp{}^+)^\top\mathcal I^{f+}
  =0
\).
Consequently, these interface actions can be represented by a
cross-element exchange operator
\(
  \widetilde G_f^{(o)}
  \in\mathbb R^{mN_p\times mN_p}
\)
such that
\[
  \mathcal I^{f-}
  =
  \widetilde G_f^{(o)}\UHATOp{}^+,
  \qquad
  \mathcal I^{f+}
  =
  -(\widetilde G_f^{(o)})^\top\UHATOp{}^- .
\]
Here, the block
\(
  (\widetilde G_f^{(o)})_{\alpha\gamma}
\)
represents the cross-element interaction
between mode \(\alpha\) in $\Omega^-$ and mode \(\gamma\) in
$\Omega^+$. Hence the corresponding modal energy contributions admit
a pairwise decomposition with equal magnitude and opposite sign
between the two elements. In the I--EIG limit,
this representation recovers the interelement
exchange of \Cref{prop:interface-conservation}.

The present construction is closely related in spirit to the
entropy-conservative interface formulation of
Tadmor~\cite{Tadmor1987,Tadmor2003Acta}; the interface
coupling is constrained by exact energy compatibility
with respect to the physical energy.
The difference is that the present formulation is built on
the $H$-metric representation, for which the projected normal
operators $G_{n,p+1}^-$ and $G_{n,p+1}^+$ are generally different on
the two sides of an interface.
Consequently, unlike the standard Tadmor construction based on a
single shared numerical state flux,
the numerical fluxes
$\widehat{\mathcal G}_{H,n}^-$ and
$\widehat{\mathcal G}_{H,n}^+$
are defined separately on the two
sides. Nevertheless, \cref{eq:cpg-interface-energy-compatibility}
ensures that they represent the same scalar physical-energy transfer
across the interface.

\subsection{Implementation in the fixed Galerkin basis}

Since the orthogonalized equation \cref{eq:h-ecf-state} includes
\(\dMHAPP\) through \(\widetilde S\), it does not provide an explicit
evolution equation for \(\UHATOp\) alone.
The system must therefore be closed through its relation to the
fixed-basis coefficients.
The fixed-basis coefficients \(\UHATp\) directly determine
\(\MHAPP\), \(\widetilde R\), and \(\UHATOp\), whereas
\(\UHATOp\) does not directly determine \(\UHATp\) because
\(\widetilde R\) itself depends on \(\UHATp\).
We therefore use \(\UHATp\) as the time-integration variable.

\begin{proposition}[Equivalent fixed-basis coefficient equation]
  \label{prop:fixed-basis-velocity}
  For a coefficient \(V\in\mathbb R^{mN_p}\),
  define the linear functional
  \[
    {\mathcal L_H}(V)
    :=
    \frac{
      \UHATp{}^\top\mathcal M_H[V]\UHATp
    }{
      2\UHATp{}^\top\MHAPP\UHATp
    },
    \qquad
    \mathcal M_H[V]
    :=
    Q_{3p}\!\left[
      \Phi^\top
      D_{\Up} H(\Up)[\Phi V]
      \Phi
    \right].
  \]
  Also, set
  \[
    F_H
    :=
    \widetilde R^{-\top}
    \left[
      (\BHAPP-\NHAPP)\UHATp-\QHNUM
    \right]
    -\gamma_H\UHATOp,
    \qquad
    \gamma_H
    :=
    \frac{
      \UHATOp{}^\top
      (\VHOAPP+\BHODAPP)
      \UHATOp
    }{
      2|\UHATOp|^2
    }.
  \]
  Then the \(H\)--ECF/CPG equation is equivalent to
  \begin{equation}
    \widetilde R\dUHATp
    +\UHATOp\,{\mathcal L_H}(\dUHATp)
    =F_H.
    \label{eq:rank-one-velocity}
  \end{equation}
\end{proposition}

\begin{proof}
  In \cref{eq:h-ecf-state}, use
  \cref{eq:true-frame-velocity} and
  \(
    \widetilde S
    =
    \widetilde S_{\rm skw}
    +\tfrac12\dMHOAPP
  \).
  Since
  \(
    \widetilde C_H^{\rm EC}\UHATOp
    =
    \tfrac12
    (\dMHOAPP+\VHOAPP+\BHODAPP)\UHATOp
    -\widetilde L_H\UHATOp
  \),
  the terms containing
  \(\tfrac12\dMHOAPP\UHATOp\) cancel.
  Moreover,
  \(
    \widetilde L_H\UHATOp
    =
    [{\mathcal L_H}(\dUHATp)+\gamma_H]\UHATOp
  \).
  The \(H\)--SBP identity gives
  \cref{eq:rank-one-velocity}.
\end{proof}

For a fixed state \(\UHATp\), \cref{eq:rank-one-velocity} is linear in
\(\dUHATp\).
Thus, its computation requires no nonlinear iteration.
Once this linear system is solved, the resulting \(\dUHATp\) is used
to advance \(\UHATp\) in time.

\section{Fixed-time construction defects of modal energy exchange}
\label{sec:exchange-defect}
This section estimates the fixed-time construction defects between the
practical \(H\)--ECF/CPG construction and its exact-integration
counterpart \(H\)--ECF/EIG.
The two constructions are evaluated at the same fixed-basis state
\(\UHATp\), the same coefficient direction
\(V\in\mathbb R^{mN_p}\), and fixed \(p\).
In this section, \(V\) represents the time-evolution direction in
coefficient space.
The resulting defects quantify differences in modal energy exchange at
the operator level, excluding differences between trajectories.
The superscript \({\rm ref}\) denotes the corresponding reference
quantity.

\subsection{Comparison frame and defect measures}
\label{subsec:comparison-frame-defects}

The practical directional MME is \(\mathcal M_H[V]\) in
\Cref{prop:fixed-basis-velocity}. Its exact-integration counterpart is
\[
  \mathcal M_H^{\rm ref}[V]
  :=
  \int_\Omega
  \Phi^\top D_{\Up} H(\Up)[\Phi V]\Phi\,dV .
  \]
The compatibility, connection, and antisymmetrized operators are
constructed using the MME along the solution trajectory. They are
\(\widetilde C_H\), \(\widetilde S\), and
\(\widetilde K_H^{\rm EC}\) for \(H\)--ECF/CPG, and
\(C_H^{\rm ref}\), \(S^{\rm ref}\), and
\(K_H^{{\rm EC},{\rm ref}}\) for \(H\)--ECF/EIG, respectively.

By \Cref{lem:hq-mass-spd}, \(\MHAPP\), constructed using a
finite-degree quadrature rule, is SPD. The same holds immediately for
\(\MHref\), defined by exact integration.
Accordingly, both constructions admit the \(H\)-orthogonalization
defined in \Cref{def:h-orth}.
However, the orthogonalized frames of the two constructions differ.
We therefore compare their exchange bilinear forms in the reference frame.
With \(Q:=\widetilde R(R^{\rm ref})^{-1}\), we have
\(\UHATOp=Q\UHATOpref\), and define
\begin{equation}
  (Y^{(o)})^\sharp:=Q^\top Y^{(o)}Q .
  \label{eq:sharp-frame}
\end{equation}
This pullback is exact and introduces no construction defect.  Define
\begin{align}
  \Delta N
  &:= (\NHOAPPSKW)^\sharp-N_{H,{\rm skw}}^{(o),{\rm ref}},\notag\\
  \Delta S
  &:= (\widetilde S_{\rm skw}[V])^\sharp-S_{\rm skw}^{\rm ref}[V],
  \qquad
  \Delta K
  := (\widetilde K_H^{\rm EC}[V])^\sharp -K_H^{{\rm EC},{\rm ref}}[V]. \notag
\end{align}

For a fixed mode index \(j\in\{1,\dots,N_p\}\), let
\(\Psi^{\rm ref}:=\Phi(R^{\rm ref})^{-1}\), and let
\(\mathsf P_j\) be the block-diagonal projector onto the \(j\)th modal
block.  The corresponding split is
\(\Up{}_j^{\rm ref}:=\Psi^{\rm ref}\mathsf P_j\UHATOpref\) and
\(\Up{}_{-j}^{\rm ref}:=\Psi^{\rm ref}(I-\mathsf P_j)\UHATOpref\).
For each fixed \(j\), \Cref{ass:up-bounds} gives uniform bounds on
these fields, their first derivatives, and their traces.  We measure
the exchange defect between the \(j\)th block and its complement by
\begin{equation}
  \mathcal E_j(Y)
  :=
  \frac{
    \left|
    (\UHATOpref)^\top
    (\mathsf P_jY_{\rm skw}-Y_{\rm skw}\mathsf P_j)
    \UHATOpref
    \right|
  }{
    |\UHATOpref|^2
  },
  \quad
  \UHATOpref\ne0 .
  \label{eq:mode-defect-functional}
\end{equation}

\subsection{Estimates of the defects}
\label{subsec:defect-estimates}
We estimate the projection defect and its contribution to \(N\),
followed by the defects in the connection and the \(H\)--ECF lift.

\begin{lemma}[Quadrature-induced projection defect]
  \label{lem:projection-defect}
  For \(X\in\{A,\Delta A\}\) and \(k\in\{1,\dots,d\}\), define
  \(
  E_k^X :=
  \Pi_{L^2,p+1}^{Q_{\mathrm{hi}(n)}}
  [H(\Up)X_k(\Up)]-
  \Pi_{L^2,p+1}^{Q_{\mathrm{hi}(\infty)}}
  [H(\Up)X_k(\Up)]
  \).
  Then there exists a positive function \(\varepsilon_X(n)\),
  independent of \(h\), such that \(\varepsilon_X(n)\to0\) as
  \(n\to\infty\) and
  \begin{align}
    \sum_{k=1}^d\|E_k^X\|_{L^2(\Omega);F}
    &\le
    \varepsilon_X(n)h^{p+2+d/2},
    \label{eq:proj-def-vol}\\
    h^{-1/2}\sum_{k=1}^d
    \|E_k^X|_{\partial\Omega}\|_{L^2(\partial\Omega);F}
    &\le
    \varepsilon_X(n)h^{p+1+d/2}.
    \label{eq:proj-def-face}
  \end{align}
\end{lemma}
\begin{proof}
  Set \(F_k^X:=H(\Up)X_k(\Up)\) and
  \(\tau_k^X:=(I-\Pi_{L^2,p+1}^{Q_{\mathrm{hi}(\infty)}})F_k^X\).
  Since both projections are the identity on
  \((V^{p+1})^{m\times m}\),
  \(E_k^X=\Pi_{L^2,p+1}^{Q_{\mathrm{hi}(n)}}\tau_k^X\).
  The projection orthogonality and exactness for
  \(E_k^X:E_k^X\) give
  \(\|E_k^X\|_{L^2(\Omega);F}^2
  =(Q_{\mathrm{hi}(n)}-Q_{\mathrm{hi}(\infty)})
  [\tau_k^X:E_k^X]\).
  The quadrature-error estimate, \Cref{ass:field-proj-error}, and
  \(\|\tau_k^X\|_{L^\infty(\Omega);F}\le Ch^{p+2}\) give
  \(\|E_k^X\|_{L^2(\Omega);F}
  \le\varepsilon_X(n)h^{p+2+d/2}\), where
  \(\varepsilon_X(n)\to0\).  Summation over \(k\) gives
  \cref{eq:proj-def-vol}, and the inverse trace inequality gives
  \cref{eq:proj-def-face}.
\end{proof}

\begin{proposition}[\(N\)-operator exchange defect]
  \label{prop:n-defect}
  For every fixed mode index \(j\), there exists a positive function
  \(C_{j,N}^{\rm hi}(n)\), independent of \(h\), such that
  \(C_{j,N}^{\rm hi}(n)\to0\) as \(n\to\infty\) and
  \begin{equation}
    \mathcal E_j(\Delta N)
    \le C_{j,N}^{\rm hi}(n)h^{p+2}.
    \label{eq:EdN-estimate}
  \end{equation}
\end{proposition}
\begin{proof}
  Both projected coefficient fields belong to
  \((V^{p+1})^{m\times m}\), so the \(Q_{3p}\)-construction is
  exact for their \(N\)-operator contributions.  Thus, \(\Delta N\)
  is generated only by \(E_k^A\).  The physical-field form of
  \cref{eq:mode-defect-functional}, together with
  \Cref{ass:up-bounds}, gives
  \(\mathcal E_j(\Delta N)\le C_jh^{-d/2}
  \sum_k\|E_k^A\|_{L^2(\Omega);F}\).
  Applying \cref{eq:proj-def-vol} proves
  \cref{eq:EdN-estimate}.
\end{proof}

\begin{proposition}[Connection and lift defects]
  \label{prop:s-k-defect}
  For every fixed mode index \(j\), there exist constants
  \(C_{j,S}^{3p},C_{j,K}^{3p}\), independent of \(h\) and \(n\),
  and a positive function \(C_{j,K}^{\rm hi}(n)\), independent of
  \(h\), such that
  \begin{equation}
    \mathcal E_j(\Delta S)
    \le C_{j,S}^{3p}h^{p+1},
    \qquad
    \mathcal E_j(\Delta K)
    \le
    \bigl(C_{j,K}^{3p}+C_{j,K}^{\rm hi}(n)\bigr)h^{p+1},
    \label{eq:s-k-defect}
  \end{equation}
  with \(C_{j,K}^{\rm hi}(n)\to0\) as \(n\to\infty\).
\end{proposition}
\begin{proof}
  Set \(G_H[V]:=D_{\Up}H(\Up)[\Phi V]\).  Inserting its exact
  \(V^p\)-projection, whose contribution is integrated exactly by
  \(Q_{3p}\), and using \Cref{ass:field-proj-error} give
  \(\|\mathcal M_H[V]-\mathcal M_H^{\rm ref}[V]\|_F
  \le C^{3p}h^{p+1}\).  Together with
  \Cref{prop:h-mass-degree} and the uniform SPD bounds, this gives
  \(\|Q-I\|_F\le Ch^{p+1}\).  The local Lipschitz continuity of
  the Cholesky map and the connection formula gives
  \(\mathcal E_j(\Delta S)\le C_{j,S}^{3p}h^{p+1}\).
  The compatibility operators contain their corresponding directional
  MMEs, whose defect was estimated above, together with the
  \(V_H\)- and \(B_{H,\Delta}\)-contributions.
  The construction quadratures are exact for the projected fields, so
  their remaining defects are generated by \(E_k^A\) and
  \(E_k^{\Delta A}\).  Using \cref{eq:proj-def-vol} with the inverse
  estimate and \cref{eq:proj-def-face} gives
  \(\|(\widetilde C_H[V])^\sharp-C_H^{\rm ref}[V]\|_F
  \le(C^{3p}+C^{\rm hi}(n))h^{p+1}\), where
  \(C^{\rm hi}(n)\to0\).  The normalized algebraic map defining the
  \(H\)--ECF lift and the pullback in \cref{eq:sharp-frame} are
  locally Lipschitz for nonzero coefficients.  Hence, the same estimate
  holds for \(\Delta K\), and applying
  \(\mathcal E_j\) proves \cref{eq:s-k-defect}.
\end{proof}

\begin{theorem}[Defect of the antisymmetric modal-energy-exchange operator]
  \label{thm:exchange-defect}
  For every fixed mode index \(j\), there exist a constant
  \(C_j^{3p}\), independent of \(h\) and \(n\), and a positive
  function \(C_j^{\rm hi}(n)\), independent of \(h\), such that
  \(C_j^{\rm hi}(n)\to0\) as \(n\to\infty\) and
  \begin{equation}
    \mathcal E_j(\Delta N-\Delta S-\Delta K)
    \le
    \bigl(C_j^{3p}+C_j^{\rm hi}(n)\bigr)h^{p+1}.
    \label{eq:exchange-defect}
  \end{equation}
\end{theorem}
\begin{proof}
  \Cref{eq:exchange-defect} follows immediately
  from \Cref{prop:n-defect,prop:s-k-defect}.
\end{proof}
Thus, at the same fixed state and coefficient direction, the practical
\(H\)--ECF/CPG exchange construction is consistent with the
\(H\)--ECF/EIG reference up to an \(O(h^{p+1})\) defect.


\section{Concluding discussion}
\label{sec:discussion}

The four models used in this paper occupy the three levels introduced
in \Cref{sec:intro}, as summarized in \Cref{tab:model-comparison}.
I--EIG is the infinite-mode reference. The finite-mode
exact-integration level contains both the native \(p\)-Galerkin EIG
model, which exposes the energy-compatibility defect, and
\(H\)--ECF/EIG, which serves as the finite-mode closure reference.
\(H\)--ECF/CPG is its practical quadrature-based realization. The
consistency path is therefore from \(H\)--ECF/CPG to \(H\)--ECF/EIG as
the quadrature defects are removed, and then to I--EIG as the resolved
mode space is enlarged.

\begin{table}[t]
  \centering \scriptsize \setlength{\tabcolsep}{3pt}
  \renewcommand{\arraystretch}{1.25}
  \caption{Structural comparison of I--EIG, native \(p\)-Galerkin EIG,
    \(H\)--ECF/EIG, and \(H\)--ECF/CPG.}
  \label{tab:model-comparison}
  \begin{tabular}{p{0.17\textwidth}
      p{0.18\textwidth} p{0.18\textwidth} p{0.18\textwidth}
      p{0.18\textwidth}} \toprule & I--EIG & EIG & \(H\)--ECF/EIG &
    \(H\)--ECF/CPG \\ \midrule Role & Continuous reference system &
    Native Galerkin model & Reference closure model & Practical
    closure model \\

    Mode space & \(V^\infty\) & \(V^p\) & \(V^p\) & \(V^p\) \\

    Operator level & Exact integration & Exact integration & Exact
    integration & CPG quadrature \\

    \(H\)--ECF & --- & No & Yes & Yes \\

    \(H\)--SBP identity & Automatic & Automatic & Automatic &
    By CPG \\

    Energy--state compatibility & Exact & Not guaranteed & Enforced by
    \(H\)--ECF & Enforced by \(H\)--ECF \\

    Equation consistency
    & Original PDE
    & Original PDE as \(p\to\infty\)
    & Original PDE as \(p\to\infty\)
    & \(H\)--ECF/EIG as the construction quadrature degree increases
    \\ \bottomrule
  \end{tabular}
\end{table}

The exchange-defect estimate in \Cref{sec:exchange-defect} should be
read in this hierarchy. The common-frame transport is an exact change
of representation and does not contribute to the
modal-energy-exchange defect. The \(N\)-defect has the simplest
origin: once the coefficient field has been projected into
\(V^{p+1}\), the \(Q_{3p}\)-construction layer integrates the projected
\(N\)-operator contribution exactly. Hence, \(\Delta N\) has no
independent construction-floor term, and its remaining contribution is
the projection defect reduced by enriching \(Q_{\mathrm{hi}(n)}\).
In contrast, the \(S\)-defect arises from the trajectory-induced
Cholesky connection and contains the \(n\)-independent floor set by
the fixed construction layer. The \(K\)-defect arises from the
\(H\)--ECF closure and contains both this floor and the projection
contribution reduced by enriching \(Q_{\mathrm{hi}(n)}\).
By \Cref{thm:exchange-defect}, this floor changes only the constant and
does not reduce the nominal \(O(h^{p+1})\) consistency of the
modal-energy-exchange operator.

The central result of this paper is that total physical-energy
conservation follows from preserving the pairwise modal
energy-exchange structure, rather than being imposed only as a scalar
balance. The \(H\)--SBP identity provides the discrete
integration-by-parts structure, while \(H\)--ECF removes only the
component incompatible with the scalar energy identity without
replacing the trajectory-induced metric evolution.
The resulting finite-mode construction is implementable
in the fixed basis and is consistent with its exact-integration reference
at a fixed state and coefficient direction.
Its relation to the fixed basis also closes
the construction as an autonomous system.


\section*{Acknowledgments}
The author acknowledges partial support from JST/JICA SATREPS
(Grant No.~JPMJSA2109), thanks Drs.~Yuta Kawai and Seiya Nishizawa
for helpful discussions, and is grateful to his wife and daughter for
their encouragement.
Large language models were used for brainstorming, structural
refinement, language editing, and preliminary organization of proof
explanations.  All mathematical statements, derivations, proofs,
technical decisions, and scientific conclusions were verified and
finalized by the author.

\bibliographystyle{siamplain}
\bibliography{references}
\end{document}